\documentclass[11pt]{amsart}
\usepackage{empheq}
\usepackage{amsfonts,amssymb,amsmath,mathrsfs,graphicx}
\usepackage{float}
\usepackage{epsfig}
\usepackage{color}
\usepackage{enumerate}
\title[Global existence of weak solutions to the Navier-Stokes equations]{Global existence of weak solutions to the Navier-Stokes equations with temperature-depending viscosity coefficient}

\author[Yu]{Cheng Yu}
\address[Cheng Yu]{\newline \newline Department of Mathematics, University of Florida, Gainesville, FL 32611, United States of America}
\email{chengyu@ufl.edu}

\author[Zuo]{Bijun Zuo}
\address[Bijun Zuo]{\newline \newline Institute of Applied Mathematics, Academy of Mathematics and Systems Science, Chinese Academy of Sciences, Beijing 100190, China
\newline and School of Mathematical Sciences, University of Chinese Academy of Sciences, Beijing 100049, China}
\email{bjzuo@amss.ac.cn}

\thanks{\textbf{Acknowledgment.} 
Cheng Yu is partially supported by Collaboration Grants for Mathematicians from Simons Foundation.
}

\newtheorem{theorem}{Theorem}[section]
\newtheorem{lemma}{Lemma}[section]

\newtheorem{proposition}{Proposition}[section]
\newtheorem{remark}{Remark}[section]

\newtheorem{definition}{Definition}[section]

\def\charf {\mbox{{\text 1}\kern-.30em {\text l}}}

\begin{document}

\date{\today}

  %
  %

\begin{abstract}
In this paper, the initial-boundary value problem to the three-dimensional inhomogeneous, incompressible and heat-conducting Navier-Stokes equations with temperature-depending viscosity coefficient is considered in a bounded domain. The viscosity coefficient is degenerate and may vanish in the region of absolutely zero temperature. Global existence of weak solutions to such a system is established for the large initial data. The proof is based on a three-level approximate scheme, the De Giorgi's method and compactness  arguments.
\end{abstract}

\maketitle \centerline{\date}
%
%
\section{Introduction and Main Result}
\setcounter{equation}{0}
In the present paper, we consider the following three-dimensional inhomogeneous, incompressible Navier-Stokes equations with the temperature effects in $(0,T)\times\Omega$:
\begin{subequations}\label{1.1}
\begin{empheq}{align}
&\partial_t\varrho+{\rm div}(\varrho\mathbf{u})=0,\\
&{\rm div}\mathbf{u}=0,\\
&\partial_t(\varrho\mathbf{u})+{\rm div}(\varrho\mathbf{u}\otimes\mathbf{u})+\nabla P-{\rm div}\mathbb{S}=0,\\
&\partial_t(\varrho\vartheta)+{\rm div}(\varrho\mathbf{u}\vartheta)+{\rm div}\mathbf{q}=\mathbb{S}:\nabla\mathbf{u},
\end{empheq}
\end{subequations}
together with the initial conditions
\begin{equation}\label{1.5}
(\varrho,\varrho\mathbf{u},\vartheta)(0,x)=(\varrho_0,\mathbf{m}_0,\vartheta_0)(x)\quad{\rm in} \,\,\Omega,
\end{equation}
and the boundary conditions
\begin{equation}\label{1.6}
\mathbf{u}(t,x)=0,\quad\nabla\vartheta(t,x)\cdot\mathbf{n}(x)=0\quad{\rm on} \,\,[0,T]\times\partial\Omega.
\end{equation}
Here $\Omega\subset \mathbb{R}^3$ is a bounded domain of class $C^{2+\nu}$ with $\nu>0$,  $\mathbf{n}(x)$ is the unit outward normal vector to the boundary at $x\in\partial\Omega$, $\varrho=\varrho(t,x)$ is the density of the fluid, $\mathbf{u}=\mathbf{u}(t,x)$ is the velocity field, $\vartheta=\vartheta(t,x)$ is the temperature, $P=P(t,x)$ is the pressure, $\mathbb{S}$ denotes the viscous stress tensor given by
\begin{equation}\label{1-01}
\mathbb{S}=\mu(\vartheta)(\nabla\mathbf{u}+\nabla^{\mathbf{t}}\mathbf{u}),
\end{equation}
where $\nabla^{\mathbf{t}}\mathbf{u}$ is the transposition of $\nabla\mathbf{u}$ and $\mu(\vartheta)\geq0$ is the viscosity coefficient which depends on the temperature and may degenerate in the region of absolutely zero temperature,
and $\mathbf{q}$ denotes the heat flux of the fluid satisfying the Fourier's law
\begin{equation}\label{1-02}
\mathbf{q}=-\kappa(\vartheta)\nabla\vartheta,
\end{equation}
with $\kappa(\vartheta)>0$ being the heat conductivity coefficient depending on the fluid temperature.

\

The existence of weak solutions globally in time for the incompressible Navier-Stokes equations was first established by  Leray \cite{Leray 1934} in 1934. The notion of weak solutions was introduced in this work, and it preceded both the introduction of the Sobolev spaces in 1936 and the generalized derivatives by Schwartz in 1944. A comparable theory for the compressible Navier-Stokes equations was developed by Lions  \cite{Lions 1998 Compressible models} where the author adopted the concept of renormalized solution developed in the framework of a joint work with DiPerna \cite{DiPerna and Lions}. In particular, Lions \cite{Lions 1998 Compressible models} proved the global existence of weak solutions with large initial data for $\gamma\geq3/2$ in two dimensions and for $\gamma\geq9/5$ in three dimensions, which was then extended to $\gamma>1$ for spherically symmetric case by Jiang and Zhang \cite{Jiang Song and Zhang Ping} and $\gamma>3/2$ for general three-dimensional case by Feireisl, Novotn\'{y} and Petzeltov\'{a} \cite{Feireisl 2001}. Furthermore, Feireisl \cite{Feireisl Dynamics 2004} extended such a result to the full compressible Navier-Stokes equations by the concept of variational solutions. By generalizing the renormalized techniques, the existence of global solutions to the compressible Navier-Stokes equations with degenerate viscosity coefficients was established in \cite{BVY, Li and Xin, Vasseur and Yu} and the existence of global solutions to the quantum Navier-Stokes equations \cite{LV}. Meanwhile, we also mention the existence of global weak or strong  solutions to the compressible Navier-Stokes equations in the absence of the vacuum states (see
\cite{Hoff 1987, Kazhikhov and Shelukhin, Serre}) and the global existence of spherically symmetric weak solutions in \cite{{Ducomet 2011}, {Guo 2008}, {Hoff 1992}}.

\

When the temperature effects are considered for the incompressible Navier-Stokes equations,
Lions \cite{Lions 1996 Incompressible models} first established the existence of global weak solutions with constant viscosity and heat conductivity coefficients by two different approaches.
Later, for the case that the viscosity
and heat conductivity coefficients depend on the temperature but they have the below positive bounds, Feireisl and M\'{a}lek \cite{Feireisl on the Navier-Stokes 2006} proved the global existence of weak solutions satisfying the total energy equation in a periodic domain, which was then extended to the general domain case by Bul\'{i}\v{c}ek, Feireisl and M\'{a}lek \cite{Feireisl a Navier-Stokes-Fourier system 2009}.
The existence of global weak solutions to \eqref{1.1}-\eqref{1.6} was proposed by Lions in  \cite{Lions 1996 Incompressible models}
where $\mu(s)\in C(\mathbb{R})$ and $\inf \{\mu(s), s\in \mathbb{R}\}>0$.
In particular, Lions pointed out that it is possible to study such a problem by the methods developed in Chapter 2 and Chapter 3 of his book \cite{Lions 1996 Incompressible models}. Essentially, this viscosity coefficient is bounded below from zero. This can provide uniform bounds for the compactness analysis.

\

As mentioned in \cite{Ba}, zero viscosity only occurs when the temperature is very low in superfluids.
Otherwise, by the second law of thermodynamics, the viscosity of all fluids is positive.
Thus, mathematically, it is natural to consider that the viscosity coefficient degenerates at very low absolute temperature, which means that the restrictions imposed on the viscosity coefficient as in \eqref{1.32} are reasonable.
Due to the lack of a positive bound from below for the viscosity coefficient,
 the uniform bounds cannot be obtained by the direct energy method, which is the first difficulty for  constructing a global weak solution to \eqref{1.1}-\eqref{1.6}.

\

The crucial point in obtaining uniform bounds is to show that the temperature is bounded away from zero by De Giorgi's method \cite{De}, which was originally developed by De Giorgi for the regularity of elliptic equations with discontinuous coefficients. Inspired by \cite{De}, Mellet and Vasseur \cite{Vasseur temperature} gave a positive bound from below for the temperature in compressible Navier-Stokes equations when the viscosity is strictly positive, and later for more general case by Baer and Vasseur \cite{BV}. Those results can be viewed as a priori estimates for the temperature and motivate us to show that the temperature is uniformly bounded away from zero for the approximated solutions. In fact, we have to add one term $\varepsilon\Delta \mathbf{u}$ in the moment equation so that we can construct approximated solutions which meet a higher regularity for the velocity field $\mathbf{u}$.
In particular, at this level approximation,
 the density, velocity and temperature fields $(\rho,\mathbf{u},\vartheta)$ satisfy a renormalized  temperature inequality \eqref{2.8} in the sense of distributions. Then, we can apply De Giorgi's method to this inequality to obtain a positive bound from below for the temperature, which further yields the lower bound for the viscosity coefficient $\mu(\vartheta)$, that is, $\mu(\vartheta)\geq \underline{\mu}>0$ for some positive constant $\underline{\mu}$. Then the uniform $H^1$-regularity of the velocity field $\mathbf{u}$ could be derived by the elementary energy inequality.

\

Another difficulty is how to handle the temperature equation.
As mentioned in \cite{{Feireisl Dynamics 2004}, {Feireisl On the motion 2004}}, one of the difficulties is the temperature concentration, which can be resolved by the renormalization of the temperature equation (\ref{1.1}d). To be specific, multiplying (\ref{1.1}d) by $h(\vartheta)$ for some suitable function $h$, we obtain
\begin{equation}\label{1-1}
\begin{split}
\partial_t(\varrho H(\vartheta))+{\rm div}(\varrho\mathbf{u}H(\vartheta))-\triangle\mathcal{K}_h(\vartheta)
=h(\vartheta)\mathbb{S}:\nabla\mathbf{u}-h'(\vartheta)\kappa(\vartheta)|\nabla\vartheta|^2,
\end{split}
\end{equation}
where
\begin{equation*}
H(\vartheta)=\int_0^\vartheta h(z) dz,\quad\quad \mathcal{K}_h(\vartheta)=\int_0^\vartheta \kappa(z)h(z) dz.
\end{equation*}
The idea of renormalization is inspired by DiPerna and Lions in \cite{DiPerna and Lions} and was used by Lions
\cite{Lions 1996 Incompressible models} and Feireisl \cite{Feireisl Dynamics 2004} to overcome the temperature concentration.  But, we are not able to recover (\ref{1.1}d) by passing to the limit  due to the regularity of velocity, instead of an inequality. This is inspired by the notation of variational solutions of Feireisl \cite{Feireisl Dynamics 2004}.

\

The weak solutions to the initial-boundary value problem \eqref{1.1}-\eqref{1.6} are defined in the following sense:

\begin{definition}\label{1.1.}
We call $(\varrho, \mathbf{u}, \vartheta, P)$ a weak solution to the initial-boundary value problem \eqref{1.1}-\eqref{1.6} if
\begin{enumerate}[(i).]

\item the density $\varrho\geq0$ and the velocity $\mathbf{u}$ satisfy
\begin{equation*}
\begin{split}
&\varrho\in L^\infty((0,T)\times\Omega)\cap C([0,T];L^p(\Omega)),\, 1\leq p<\infty,\\
&\mathbf{u}\in L^2(0,T; H^1_0(\Omega)),\quad\sqrt{\varrho}\mathbf{u}\in L^\infty(0,T; L^2(\Omega)),
\end{split}
\end{equation*}
and the temperature $\vartheta\geq 0$ satisfies
\begin{equation*}
\vartheta \in L^2(0,T; H^1(\Omega)),\quad \varrho\vartheta\in L^\infty(0,T; L^1(\Omega));
\end{equation*}

\item the equations (\ref{1.1}a)-(\ref{1.1}c) hold in $\mathcal{D}'((0,T)\times\Omega)$, that is, for any $\Phi\in C^\infty_c((0,T)\times\Omega)$ satisfying ${\rm div}\Phi=0$, we have
\begin{equation}\label{1.12}
\int_0^T\int_\Omega \varrho \partial_t\Phi dxdt+\int_0^T\int_\Omega \varrho\mathbf{u}\cdot\nabla \Phi dxdt=0,
\end{equation}
\begin{equation}\label{1.14}
\begin{split}
&\int_0^T\int_\Omega \varrho\mathbf{u}\cdot \partial_t\Phi dxdt+\int_0^T\int_\Omega \varrho\mathbf{u}\otimes\mathbf{u}:\nabla\Phi dxdt
=\int_0^T\int_\Omega \mathbb{S}:\nabla \Phi  dxdt,
\end{split}
\end{equation}
 and for any $\eta\in C^\infty_c(\Omega)$, it holds that for a.e. $t\in (0,T)$
 \begin{equation*}
 \int_\Omega \mathbf{u}\cdot\nabla\eta dx=0;
 \end{equation*}


\item  $\vartheta(t,\cdot)\rightarrow \vartheta_{0}$,\,\,$\varrho\mathbf{u}(t,\cdot)\rightarrow \mathbf{m}_{0}$ in $\mathcal{D}'(\Omega)$ as $t\rightarrow 0^+$, that is, for any $\chi\in C_c^\infty(\Omega)$, it holds
\begin{equation*}
\begin{split}
&\lim_{t\rightarrow 0^+}\int_{\Omega}\vartheta(t,x)\chi(x) dx=\int_{\Omega}\vartheta_0(x)\chi(x)dx,\\
&\lim_{t\rightarrow 0^+}\int_{\Omega}\varrho\mathbf{u}(t,x)\chi(x) dx=\int_{\Omega}\mathbf{m}_0(x)\chi(x)dx;
\end{split}
\end{equation*}

\item the following temperature inequality holds
\begin{equation}\label{1.16}
\begin{split}
&\int_0^T\int_\Omega \varrho\vartheta\partial_t\varphi dxdt+\int_0^T\int_\Omega \left(\varrho\mathbf{u}\vartheta\cdot\nabla\varphi
+\mathcal{K}(\vartheta)\triangle\varphi\right) dxdt\\
&\leq -\int_0^T\int_\Omega \mathbb{S}:\nabla\mathbf{u}\varphi dxdt
-\int_\Omega \varrho_0\vartheta_0\varphi(0),
\end{split}
\end{equation}
where $\varrho\mathbb{S}=\varrho\mu(\vartheta)(\nabla\mathbf{u}+\nabla^{\mathbf{t}}\mathbf{u})$ and
$\mathcal{K}(\vartheta)=\int_0^\vartheta \kappa(z) dz$,
for any $\varphi\in C_c^\infty([0,T]\times\Omega)$ satisfying
\begin{equation}\label{1.17}
\varphi\geq 0,\,\,\varphi(T,\cdot)=0,\,\,\nabla\varphi\cdot \mathbf{n}|_{\partial\Omega}=0;
\end{equation}

\item the energy inequality holds, that is, for a.e. $t\in(0,T)$,
\begin{equation}\label{1.18}
E[\varrho, \mathbf{u}, \vartheta](t)\leq E[\varrho, \mathbf{u}, \vartheta](0),
\end{equation}
where
\begin{equation}\label{1.180}
E[\varrho, \mathbf{u}, \vartheta](t)=\int_\Omega \left(\frac{1}{2}\varrho|\mathbf{u}|^2+\varrho\vartheta\right)(t) dx,
\end{equation}
and
\begin{equation}\label{1-2}
E[\varrho, \mathbf{u}, \vartheta](0)=\int_\Omega \left(\frac{1}{2}\frac{|\mathbf{m}_0|^2}{\varrho_0}+\varrho_0\vartheta_0\right)dx.
\end{equation}
\end{enumerate}
\end{definition}

\

\begin{remark}\label{1.4.}
As mentioned in \cite{Feireisl on the Navier-Stokes 2006}, the reason
for introducing the function $\mathcal{K}(\vartheta)=\int_0^\vartheta \kappa(z) dz$
with $\nabla\mathcal{K}(\vartheta)=\kappa(\vartheta)\nabla\vartheta=-\mathbf{q}$
is that we are unable to deduce $\kappa(\vartheta)\nabla\vartheta$ is
locally integrable by a priori estimates that we can obtain. However, we can deduce $\mathcal{K}(\vartheta)\in L^1((0,T)\times\Omega)$ by
constructing proper approximate equations and requiring suitable growth restrictions on $\kappa(\vartheta)$.
\end{remark}

\begin{remark}\label{1.3.}
As shown later, the bounds on the velocity fail to ensure the convergence of the term $\mathbb{S}:\nabla\mathbf{u}$ in the sense of distributions. Therefore, as in \cite{Feireisl Dynamics 2004, Feireisl On the motion 2004}, we replaced the temperature equation (\ref{1.1}d) by the following two inequalities in Definition \ref{1.1.}
\begin{equation}\label{1.19}
\partial_t(\varrho\vartheta)+{\rm div}(\varrho\mathbf{u}\vartheta)-\triangle\mathcal{K}(\vartheta)\geq\mathbb{S}:\nabla\mathbf{u},
\end{equation}
and
\begin{equation}\label{1.20}
E[\varrho, \mathbf{u}, \vartheta](t)\leq E[\varrho, \mathbf{u}, \vartheta](0).
\end{equation}
\end{remark}

\

Our main result can be stated as follows.
\begin{theorem}\label{1.2.}
Let $\Omega\subset \mathbb{R}^3$ be a bounded domain of class $C^{2+\nu}$, with $\nu>0$. Assume that
\begin{enumerate}[(i).]
\item the heat conductivity coefficient $\kappa(\vartheta)\in C^1([0,\infty))$ satisfies
\begin{equation}\label{1.33}
\underline{\kappa}(1+\vartheta^2)\leq\kappa(\vartheta)\leq\overline{\kappa}(1+\vartheta^2),
\end{equation}
for constants $\underline{\kappa}>0$ and $\overline{\kappa}>0$;

\item the viscosity coefficient $\mu(\vartheta)$ is globally Lipschitz continuous on $[0,\infty)$ and it is a positive function on $[\bar{\vartheta},\infty)$, satisfying
\begin{equation}\label{1.32}
\lim_{\vartheta\to\infty}\mu(\vartheta)>0,\,\,\text{ and }\,\,
 \mu(\vartheta)\geq \kappa \vartheta,\;\;\text{ for } 0\leq \vartheta\leq \bar{\vartheta};
\end{equation}

\item the initial data satisfy
\begin{equation}\label{1.30}
\begin{cases}
\varrho_0\in L^\infty(\Omega), \,\, \varrho_0\geq 0 \,\,&{\rm on}\,\,\Omega,\\
\vartheta_0\in L^1(\Omega),\,\,\vartheta_0\geq\underline{\vartheta}>0 \,\,& {\rm on}\,\,\Omega,\\
\frac{|\mathbf{m}_0|^2}{\varrho_0}\in L^1(\Omega).
\end{cases}
\end{equation}
\end{enumerate}
Then, for any given $T>0$, the initial-boundary value problem \eqref{1.1}-\eqref{1.6} admits a global weak solution $(\varrho,\mathbf{u},\vartheta, P)$ in the sense of Definition \ref{1.1.}.
\end{theorem}

\begin{remark}
Zero viscosity only occurs when the temperature is very low in superfluids.
Otherwise, by the second law of thermodynamics, the viscosity of all fluids is positive.
Thus, our restriction \eqref{1.32} is physical.
\end{remark}

The proof of Theorem \ref{1.2.}
will be divided into the following three steps:

Step 1: For fixed $\varepsilon$, $\delta>0$, we solve the approximate system \eqref{1.50}-\eqref{1.55} by the Galerkin method, to be specific, we first solve the problem in a suitable finite dimensional space $X_n$, then recover a global solution by passing to the limit as $n\rightarrow\infty$.

Step 2: For fixed $\delta>0$, letting $\varepsilon\rightarrow 0$ to eliminate the artificial viscosity in the momentum equation (\ref{1.50}c). To this end, the crucial point is to obtain a below bound for the temperature by De Giorgi's method, which is motivated by the work of Mellet, Vasseur \cite{Vasseur temperature}. This implies that the temperature-depending viscosity coefficient is bounded away from zero by the assumptions in Theorem \ref{1.2.}. Thus, one obtains the uniform bounds for the velocity for passing to the limits.
Note that, the below bound of temperature is uniform in $\delta$. So the uniform bounds are also available for the limits as $\delta$ goes to zero.

Step 3: We are able to recover a globally defined weak solution to the initial-boundary value problem \eqref{1.1}-\eqref{1.6} by letting $\delta\to 0$.

\

Our paper is organized as follows.
In Section 2, we construct a suitable approximate system \eqref{1.50}-\eqref{1.55}
and obtain the global solvability by the Galerkin method.
In Section 3, we perform the limit $\varepsilon\rightarrow 0$ to eliminate the artificial viscosity.
In Section 4, we let $\delta\rightarrow 0$ to finish the proof of Theorem \ref{1.2.}.

\

\section{The construction of approximate solutions}
\setcounter{equation}{0}
First, we construct the following approximate system in $(0,T)\times\Omega$:
\begin{subequations}\label{1.50}
\begin{empheq}{align}
&\partial_t\varrho+{\rm div}(\varrho\mathbf{u})=0,\\
&{\rm div}\mathbf{u}=0,\\
&\partial_t(\varrho\mathbf{u})+{\rm div}(\varrho\mathbf{u}\otimes\mathbf{u})+\nabla P-{\rm div}\mathbb{S}-\varepsilon\triangle\mathbf{u}=0,\\
&\partial_t((\delta+\varrho)\vartheta)+{\rm div}(\varrho\mathbf{u}\vartheta)-\triangle \mathcal{K}(\vartheta)
+\delta \vartheta^3=(1-\delta)\mathbb{S}:\nabla\mathbf{u},
\end{empheq}
\end{subequations}
where both $\varepsilon$ and $\delta$ are positive parameters,
supplemented with the initial conditions
\begin{equation}\label{1.54}
(\varrho,\varrho\mathbf{u},\vartheta)(0,x)
=(\varrho_ {0,\delta},\mathbf{m}_{0,\delta},\vartheta_{0,\delta})(x) \quad{\rm in}\,\,\Omega,
\end{equation}
and the boundary conditions
\begin{equation}\label{1.55}
\mathbf{u}(t,x)=0,\quad\nabla\vartheta(t,x)\cdot\mathbf{n}(x)=0\quad\;{\rm on} \,\,{[0,T]\times\partial\Omega}.
\end{equation}

Note that the construction of the above approximate system is motivated but different from
\cite{Feireisl 2001}-\cite{Feireisl Singular}.
Moreover, the regularized initial data are required to satisfy the following conditions:
\begin{equation}\label{1.56}
\begin{cases}
&\varrho_{0,\delta}\in C^{2+\nu}(\bar{\Omega}),\quad 0<\delta\leq\varrho_{0,\delta}(x)\leq \bar{\varrho};\\
&\varrho_{0,\delta}\rightarrow\varrho_{0}\,\,{\rm in}\,\,L^2(\Omega),\quad
|\{x\in\Omega\,|\,\varrho_{0,\delta}(x)<\varrho_0(x)\}|\rightarrow 0,\,\,{\rm as}\,\,\delta\rightarrow 0;\\
&\vartheta_{0,\delta}\in C^{2+\nu}(\bar{\Omega}),\quad \nabla\vartheta_{0,\delta}\cdot\mathbf{n}|_{\partial\Omega}=0,\quad
0<\underline{\vartheta}\leq\vartheta_{0,\delta};\\
&\vartheta_{0,\delta}\rightarrow\vartheta_{0}\,\,{\rm in}\,\,L^2(\Omega),\,\, {\rm as}\,\,\delta\rightarrow 0;\\
&\mathbf{m}_{0,\delta}=
\begin{cases}
\mathbf{m}_0,\,\,& {\rm if}\,\,\varrho_{0,\delta}\geq\varrho_0,\\
0,\,\,& {\rm if}\,\,\varrho_{0,\delta}<\varrho_0,
\end{cases}
\end{cases}
\end{equation}
where $\bar{\varrho}$ and $\underline{\vartheta}$ are independent of $\delta>0$.
In particular, the regularized initial value of the total energy
\begin{equation}\label{initial}
E_\delta(0)=\int_\Omega \left(\frac{1}{2}\frac{|\mathbf{m}_{0,\delta}|^2}{\varrho_{0,\delta}}
+(\delta+\varrho_{0,\delta})\vartheta_{0,\delta}\right)dx
\end{equation}
is bounded by a constant independent of $\delta>0$.

\

\begin{remark}\label{1.5.}
In equations (\ref{1.50}c) and (\ref{1.50}d), the quantities $\varepsilon$ and $\delta$ are small positive parameters. Roughly speaking, the extra term $\varepsilon\triangle\mathbf{u}$ represents the artificial viscosity which ensures the parabolic property of the momentum equation (\ref{1.50}c). The quantity $\delta\vartheta^3$ is introduced to improve the integrability of the temperature. The other terms related to the parameter $\delta>0$ are introduced to avoid technicalities in the temperature estimates.
\end{remark}

\

We give our result about the global solvability of the approximate problem \eqref{1.50}-\eqref{1.55} in the following proposition.

\begin{proposition}\label{2.1.}
For any fixed $\varepsilon$, $\delta>0$, under the hypotheses of Theorem \ref{1.2.} and the assumptions imposed on the initial data \eqref{1.56}, the approximate system \eqref{1.50}-\eqref{1.55} admits a global weak solution $(\varrho, \mathbf{u}, \vartheta, P)$ satisfying the following properties:
\begin{enumerate}[(i).]
\item the density $\varrho\geq 0$ and the velocity $\mathbf{u}$ satisfy
\begin{equation*}
\begin{split}
&\varrho\in L^\infty((0,T)\times\Omega)\cap C([0,T];L^p(\Omega)),\, 1\leq p<\infty,\\
&\mathbf{u}\in L^2(0,T; H^1_0(\Omega)),\quad\sqrt{\varrho}\mathbf{u}\in L^\infty(0,T; L^2(\Omega)),
\end{split}
\end{equation*}
and the temperature $\vartheta\geq 0$ satisfies
\begin{equation*}
\vartheta \in L^2(0,T; H^{1}(\Omega))\cap L^3((0,T)\times\Omega),\quad \varrho\vartheta\in L^\infty(0,T; L^1(\Omega));
\end{equation*}

\item the equations (\ref{1.50}a)-(\ref{1.50}c) hold in $\mathcal{D}'((0,T)\times\Omega)$;

\item  $\vartheta(t,\cdot)\rightarrow \vartheta_{0,\delta}$, $\varrho\mathbf{u}(t,\cdot)\rightarrow \mathbf{m}_{0,\delta}$ in $\mathcal{D}'(\Omega)$ as $t\rightarrow 0^+$;

\item the renormalized temperature inequality holds in the sense of distributions, that is,
\begin{equation}\label{2.8}
\begin{split}
&\int_0^T\int_\Omega (\delta+\varrho)H(\vartheta)\partial_t\varphi dxdt\\
&\quad+\int_0^T\int_\Omega \left(\varrho H(\vartheta)\mathbf{u}\cdot\nabla\varphi
+\mathcal{K}_h(\vartheta)\triangle\varphi-\delta\vartheta^3 h(\vartheta)\varphi\right) dxdt\\
&\leq\int_0^T\int_\Omega \left((\delta-1)\mathbb{S}:\nabla\mathbf{u}h(\vartheta)+h'(\vartheta)\kappa(\vartheta)|\nabla\vartheta|^2\right)\varphi dxdt\\
&\quad\quad-\int_\Omega(\delta+\varrho_{0,\delta})H(\vartheta_{0,\delta})\varphi(0)dx
\end{split}
\end{equation}
for any $\varphi\in C_c^\infty([0,T]\times\Omega)$ satisfying
\begin{equation}\label{2.9}
\varphi\geq 0,\,\,\varphi(T,\cdot)=0,\,\,\nabla\varphi\cdot \mathbf{n}|_{\partial\Omega}=0,
\end{equation}
where $H(\vartheta)=\int_0^\vartheta h(z) dz$ and $\mathcal{K}_h(\vartheta)=\int_0^\vartheta \kappa(z)h(z) dz$,
with the non-increasing $h\in C^2([0,\infty))$ satisfying
\begin{equation}\label{2.10}
0<h(0)<\infty,\,\,\lim_{z\rightarrow\infty}h(z)=0,\\
\end{equation}
and
\begin{equation}\label{2.01}
h''(z)h(z)\geq 2(h'(z))^2\,\,for \,\,all\,\,z\geq0;
\end{equation}

\item the energy inequality
\begin{equation}\label{2.11}
\begin{split}
&\int_0^T\int_\Omega(-\partial_t\psi)\left(\frac{1}{2}\varrho|\mathbf{u}|^2+(\delta+\varrho)\vartheta \right)dxdt\\
&\quad+\int_0^T\int_\Omega \psi\left(\delta\mathbb{S}:\nabla\mathbf{u}+\varepsilon|\nabla\mathbf{u}|^2+\delta \vartheta^3\right) dxdt\\
&\leq\int_\Omega \left(\frac{1}{2}\frac{|\mathbf{m}_{0}|^2}{\varrho_{0,\delta}}+(\delta+\varrho_{0,\delta})\vartheta_{0,\delta} \right)dx
\end{split}
\end{equation}
holds for any $\psi\in C^\infty([0,T])$ satisfying
\begin{equation}\label{2.1100}
\psi(0)=1,\quad \psi(T)=0,\quad \partial_t\psi\leq 0.
\end{equation}
\end{enumerate}
\end{proposition}

\

\begin{remark}\label{1.6.}
As proved in \cite{Feireisl Dynamics 2004} for the constant viscosity coefficient case, the hypothesis \eqref{2.01} is imposed to ensure the convex and weakly lower semi-continuous property of the function
\begin{equation}\label{1-10}
(\vartheta, \nabla\mathbf{u})\mapsto h(\vartheta)\mathbb{S}:\nabla\mathbf{u},
\end{equation}
which is still valid for the temperature-depending viscosity coefficient case (cf. \cite{Xianpeng Hu and Dehua Wang}).
\end{remark}

\

\subsection{\bf Global solvability of the approximate system in a finite dimensional space}\

Let $\{\eta_n\}$ be a family of divergence-free linearly independent smooth vector functions and vanish on the boundary $\partial\Omega$.
Now, consider a sequence of finite dimensional spaces
\begin{equation}\label{2.29}
X_n={\rm span}\{\eta_1, \eta_2,...,\eta_n\}, \,n=1,2,\cdot\cdot\cdot.
\end{equation}

\

The global solvability of the approximate problem \eqref{1.50}-\eqref{1.55} in the finite dimensional space $X_n$ can be achieved by the following four steps:

Step 1: Given $\mathbf{u}=\mathbf{u}_n\in C([0,T]; X_n)$, the approximate continuity equation (\ref{1.50}a) can be seen as a transport equation of $\varrho$, which can be solved directly by the characteristics method. We denote the solution by $\varrho_n:=\varrho[\mathbf{u}_n]$ and give the details in Proposition \ref{2.2.}.

Step 2: Given $\mathbf{u}=\mathbf{u}_n$ and $\varrho=\varrho_n$, the approximate temperature equation (\ref{1.50}d) can be seen as a quasi-linear parabolic equation of $\vartheta$, which can be solved by applying the parabolic theory \cite{Ladyzhenskaya}. See details in Proposition \ref{2.3.}. Denote the solutions by $\vartheta_n:=\vartheta[\mathbf{u}_n]$.

Step 3: Substituting $\varrho=\varrho_n$ and $\vartheta=\vartheta_n$ into the following integral
equation
\begin{equation}\label{2.41}
\begin{split}
\int_\Omega (\varrho\mathbf{u}_n)(t)\cdot \eta dx-\int_\Omega \mathbf{m}_{0,\delta}\cdot \eta dx
=\int_0^t\int_\Omega \left(\varrho\mathbf{u}_n\otimes\mathbf{u}_n-\mathbb{S}_n-\varepsilon\nabla\mathbf{u}_n\right):\nabla\eta dxds,
\end{split}
\end{equation}
for any $\eta\in X_n$, with \[\mathbb{S}_n=\mu(\vartheta)(\nabla\mathbf{u}_n+\nabla^{\mathbf{t}}\mathbf{u}_n),\]
we can obtain a local solution $\mathbf{u}_n\in C([0,T_n]; X_n)$ with $T_n\leq T$ by the standard fixed point theorem.
See details in Proposition \ref{2.4.}.

Step 4: By virtue of some uniform (in time) estimates, we can extend $T_n$ to $T$ to obtain a global existence result.

\begin{remark}\label{2.10.}
Note that the integral equation \eqref{2.41} can be seen as a projection of the momentum equation (\ref{1.50}c) onto the finite dimensional space $X_n$ in the sense of distributions.
\end{remark}

\


Following the steps presented above, for fixed velocity field $\mathbf{u}=\mathbf{u}_n$, we first study the solvability of the approximate continuity equation (\ref{1.50}a).

\begin{proposition}\label{2.2.}
Let $\mathbf{u}=\mathbf{u}_n$ be a given vector function belonging to $C([0,T]; X_n)$. Assume that the initial data $\varrho_{0,\delta}$ satisfy the hypotheses in \eqref{1.56}.

Then there exists a mapping $\varrho_n=\varrho[\mathbf{u}_n]${\rm:}
\begin{equation*}
\varrho_n:\,C([0,T]; X_n)\rightarrow C([0,T];C^2(\bar{\Omega}))
\end{equation*}
having the following properties:
\begin{itemize}
\item the initial value problem (\ref{1.50}a), \eqref{1.54} possesses a unique classical solution $\varrho_n$;

\item $0<\delta\leq\varrho[\mathbf{u}_n](t,x)\leq\overline{\varrho}$\, for all $t\in[0,T]$;

\item continuity of the mapping:
\begin{equation}\label{2.50}
\left\|\varrho_{n_1}-\varrho_{n_2}\right\|_{C([0,T]; C^2(\bar{\Omega}))}
\leq CT\| \mathbf{u}_{n_1}-\mathbf{u}_{n_2}\|_{C([0,T]; X_n)}.
\end{equation}
\end{itemize}
\end{proposition}

\begin{proof}
Taking $\mathbf{u}=\mathbf{u}_n$, we can rewrite the continuity equation (\ref{1.50}a) as the following transport equation
\begin{equation*}
\partial_t\varrho+\mathbf{u}_n\cdot\nabla\varrho=0.
\end{equation*}
By characteristics method, we have
\begin{equation}\label{2.52}
\varrho(t,x)=\varrho_{0,\delta}(x-\mathbf{u}_nt),
\end{equation}
which, combined with the assumptions imposed on the initial data $\varrho_{0,\delta}$ in \eqref{1.56}, yields the properties in Proposition \ref{2.2.}.
\end{proof}

\

For the temperature equation (\ref{1.50}d), similarly as in \cite{Feireisl Dynamics 2004,Feireisl Singular}, we have the following proposition.

\begin{proposition}\label{2.3.}
Let $\mathbf{u}=\mathbf{u}_n$ be a given vector function belonging to $C([0,T]; X_n)$ and $\varrho=\varrho_n$ be the unique solution in Proposition \ref{2.2.}. Suppose that the initial data $\vartheta_{0,\delta}$ satisfy the hypotheses in \eqref{1.56}.

Then there exists a mapping $\vartheta_n=\vartheta[\mathbf{u}_n]$ having the following properties:
\begin{itemize}
\item the initial-boundary value problem (\ref{1.50}d), \eqref{1.54} and \eqref{1.55} admits a unique strong solution $\vartheta_n=\vartheta[\mathbf{u}_n]$;

\item the solution $\vartheta_n$ has the following regularity properties:
\begin{equation}\label{2.60}
\begin{split}
\nabla\vartheta_n\in L^2((0,T)\times\Omega),\quad
\partial_t\vartheta_n\in L^2((0,T)\times\Omega);
\end{split}
\end{equation}

\item continuity of the mapping:
\begin{equation}\label{2.61}
\|\vartheta_{n_1}-\vartheta_{n_2}\|_{L^2(0,T;H^1(\Omega))}
\leq C\sqrt{T}\|\mathbf{u}_{n_1}-\mathbf{u}_{n_2}\|_{C([0,T];X_n)}.
\end{equation}
\end{itemize}
\end{proposition}

\

Based on Proposition \ref{2.2.} and Proposition \ref{2.3.}, we can follow the same idea as in  \cite{Feireisl Dynamics 2004}- \cite{Feireisl Singular} to obtain the local existence as follows.

\begin{proposition}\label{2.4.}
For fixed $\varepsilon, \delta>0$, assume that the initial data satisfy \eqref{1.56} and $X_n$ is defined by \eqref{2.29}. Denote $\mathbf{u}_n(0):=\mathbf{u}_{0,\delta,n}$ and suppose
$\varrho_{0,\delta}\mathbf{u}_{0,\delta,n}=\mathbf{m}_{0,\delta}$ for any n.

Then the approximate problem (\ref{1.50}a),(\ref{1.50}d), \eqref{1.54}, \eqref{1.55} and \eqref{2.41} admits a local solution $(\varrho_n, \mathbf{u}_n, \vartheta_n)$ on a short time interval $[0,T_n]$ with $T_n\leq T$ satisfying
Proposition \ref{2.2.} and Proposition \ref{2.3.}.
\end{proposition}

\

Now, in order to show $T_n=T$ for any $n$, it is enough to obtain uniform (in time) bounds on the norm $\|\mathbf{u}_n(t)\|_{X_n}$ for $t\in[0, T_n]$ independent of $T_n$, which is often obtained by energy estimates.

First, by \eqref{2.41}, we obtain that the velocity $\mathbf{u}_n$ is continuously differentiable, which implies the following integral identity holds on $(0,T_n)$ for any $\eta\in X_n$
\begin{equation}\label{4.0}
\int_\Omega \partial_t(\varrho_n\mathbf{u}_n)\cdot\eta dx
=\int_\Omega \left(\varrho_n\mathbf{u}_n\otimes\mathbf{u}_n-\mathbb{S}_n-\varepsilon\nabla\mathbf{u}_n\right):\nabla\eta dx.
\end{equation}
Taking $\eta=\mathbf{u}_n$ in \eqref{4.0}, we obtain
\begin{equation}\label{4.1}
\frac{d}{dt}\int_\Omega\frac{1}{2}\varrho_n|\mathbf{u}_n|^2 dx
+\int_\Omega\frac{1}{2}\mu(\vartheta_n)\left(\nabla\mathbf{u}_n+\nabla^{\mathbf{t}}\mathbf{u}_n\right)^2dx
+\varepsilon\int_\Omega|\nabla\mathbf{u}_n|^2 dx=0.
\end{equation}
Integrating \eqref{4.1} over $(0,\tau)$ for any $\tau\in [0,T_n]$, we have
\begin{equation}\label{4.100}
\begin{split}
&\frac{1}{2}\int_\Omega(\varrho_n|\mathbf{u}_n|^2)(\tau) dx
+\int_0^\tau\int_\Omega\frac{1}{2}\mu(\vartheta_n)\left(\nabla\mathbf{u}_n+\nabla^{\mathbf{t}}\mathbf{u}_n\right)^2dxds
+\varepsilon\int_0^\tau\int_\Omega|\nabla\mathbf{u}_n|^2 dxds\\
&=\frac{1}{2}\int_\Omega\mathbf{m}_{0,\delta}\cdot\mathbf{u}_{n}(0) dx,
\end{split}
\end{equation}
where the term on the right-hand side can be controlled by
\begin{equation}\label{4.101}
\begin{split}
&\int_\Omega\mathbf{m}_{0,\delta}\cdot\mathbf{u}_{n}(0) dx
=\int_\Omega\mathbf{m}_{0,\delta}\cdot\mathbf{u}_{0,\delta,n} dx\\
&\leq \frac{1}{2}\int_\Omega \left(\frac{|\mathbf{m}_{0,\delta}|^2}{\varrho_{0,\delta}}+\varrho_{0,\delta}|\mathbf{u}_{0,\delta,n}|^2\right)dx
=\frac{1}{2}\int_\Omega
\left(\frac{|\mathbf{m}_{0,\delta}|^2}{\varrho_{0,\delta}}+\mathbf{m}_{0,\delta}\cdot\mathbf{u}_{0,\delta,n}\right)dx.
\end{split}
\end{equation}
Thus, we deduce for any $\tau\in [0,T_n]$
\begin{equation}\label{4.102}
\begin{split}
&\frac{1}{2}\int_\Omega(\varrho_n|\mathbf{u}_n|^2)(\tau) dx
+\int_0^\tau\int_\Omega\frac{1}{2}\mu(\vartheta_n)\left(\nabla\mathbf{u}_n+\nabla^{\mathbf{t}}\mathbf{u}_n\right)^2dxds
+\varepsilon\int_0^\tau\int_\Omega|\nabla\mathbf{u}_n|^2 dxds\\
&\leq\frac{1}{2}\int_\Omega \frac{|\mathbf{m}_{0,\delta}|^2}{\varrho_{0,\delta}} dx.
\end{split}
\end{equation}
This implies
\begin{equation}\label{4.2}
\|\sqrt{\varrho_n}\mathbf{u}_n\|_{L^\infty(0,T_n; L^2(\Omega))}\leq C,
\end{equation}
where $C$ is independent of $n$ and $T_n$.

Since $\varrho_n$ is bounded from below by a positive constant, we deduce
\begin{equation}\label{4.3}
\|\mathbf{u}_n\|_{L^\infty(0,T_n; L^2(\Omega))}\leq C.
\end{equation}
By virtue of the fact that all norms are equivalent on $X_n$, we have
\begin{equation}\label{4.4}
\|\mathbf{u}_n\|_{L^\infty(0,T_n; X_n)}\leq C,
\end{equation}
with $C$ independent of $n$ and $T_n$, which allows to extend the local existence result on $T_n$ to the global existence on $T$ by repeating the fixed point argument above after finite steps.

\

\subsection{\bf Passing to the limit for $n\to\infty$}\
\

The goal of this subsection is to pass limit to the approximate solutions $(\varrho_n, \mathbf{u}_n,\vartheta_n)$ to recover a weak solution to the approximate system \eqref{1.50}-\eqref{1.55}, as $n\to\infty$, for any fixed $\varepsilon$, $\delta>0$.
For convenience, in the rest of this subsection, we denote $C$ a generic positive constant which is independent of $n$. In fact, we have the following uniform bounds.
\begin{proposition}\label{4.2.}
For fixed $\varepsilon$, $\delta>0$, under the hypotheses of Proposition \ref{2.1.}, we have
\begin{equation}\label{4.37}
\|\varrho_n\|_{L^\infty((0,T)\times\Omega)}\leq C,
\end{equation}
\begin{equation}\label{4.38}
\|\mathbf{u}_n\|_{L^2(0,T; H^{1}_0(\Omega))}\leq C,
\end{equation}
\begin{equation}\label{4.39}
\|\vartheta_n\|_{L^2(0,T; H^{1}(\Omega))}\leq C,
\end{equation}
\begin{equation}\label{4.40}
\|\vartheta_n\|_{L^3((0,T)\times\Omega)}\leq C,
\end{equation}
\begin{equation}\label{4.41}
\|\sqrt{\varrho}_n\mathbf{u}_n\|_{L^\infty(0,T; L^2(\Omega))}\leq C,
\end{equation}
\begin{equation}\label{4.42}
\|\varrho_n\vartheta_n \|_{L^\infty(0,T;L^1(\Omega))}\leq C.
\end{equation}
\end{proposition}

\begin{proof}
First, replacing $T_n$ by $T$ in the energy inequality \eqref{4.102} and thanks to Poincar\'{e}'s inequality, we have \eqref{4.38} and \eqref{4.41}. As proved in \cite{Lions 1996 Incompressible models}, by the divergence-free property of $\mathbf{u}_n$, we have for all $0\leq \alpha\leq \beta<\infty$
\[meas\{x\in\Omega\,|\alpha\leq\varrho_n(t,x)\leq\beta\}\,\,{\rm is\,\,independent\,\,of\,\,}t\geq0,\]
which combined with $0<\delta\leq\varrho_{0,\delta}(x)\leq \bar{\varrho}$ implies \eqref{4.37}.

Then, integrating the equation (\ref{1.50}d) over $(0,\tau)\times\Omega$ for any $\tau\in[0,T]$ and adding to \eqref{4.102}, we obtain the following total energy inequality
\begin{equation}\label{4.23}
\begin{split}
&\int_\Omega\left(\frac{1}{2}\varrho_n|\mathbf{u}_n|^2+(\delta+\varrho_n)\vartheta_n \right)(\tau)dx
+\varepsilon\int_0^\tau \int_\Omega|\nabla\mathbf{u}_n|^2 dxds\\
&+\delta\int_0^\tau\int_\Omega \left(\vartheta_n^3+\mathbb{S}_n:\nabla\mathbf{u}_n \right)dx ds
\leq\int_\Omega\left(\frac{1}{2}\frac{|\mathbf{m}_{0,\delta}|^2}{\varrho_{0,\delta}}
+(\delta+\varrho_{0,\delta})\vartheta_{0,\delta} \right)dx,
\end{split}
\end{equation}
which implies \eqref{4.40} and \eqref{4.42}.

Next, multiplying (\ref{1.50}d) by $h(\vartheta_n)$, we have
\begin{equation}\label{4.30}
\begin{split}
&\partial_t((\delta+\varrho_n)H(\vartheta_n))
+{\rm div}(\varrho_n \mathbf{u}_n H(\vartheta_n))
-\triangle \mathcal{K}_h(\vartheta_n)\\
&+\kappa(\vartheta_n)h'(\vartheta_n)|\nabla\vartheta_n|^2
+\delta \vartheta_n^3 h(\vartheta_n)
=(1-\delta)\mathbb{S}_n:\nabla\mathbf{u}_nh(\vartheta_n),
\end{split}
\end{equation}
where $H(\vartheta)=\int_0^{\vartheta} h(z) dz$ and $\mathcal{K}_h(\vartheta)=\int_0^{\vartheta} \kappa(z)h(z) dz$, with $h$ meeting \eqref{2.10} and \eqref{2.01}.
Integrating \eqref{4.30} over $(0,\tau)\times\Omega$ for any $\tau\in[0,T]$, one obtains
\begin{equation}\label{4.32}
\begin{split}
&\int_\Omega \left((\delta+\varrho_n)H(\vartheta_n)\right)(\tau) dx
+\int_0^\tau\int_\Omega \kappa(\vartheta_n)h'(\vartheta_n)|\nabla\vartheta_n|^2dxds\\
&\quad+ \delta \int_0^\tau\int_\Omega \vartheta_n^3h(\vartheta_n) dxds\\
&=(1-\delta)\int_0^\tau\int_\Omega\mathbb{S}_n:\nabla\mathbf{u}_nh(\vartheta_n)dxds
+\int_\Omega (\delta+\varrho_{0,\delta})H(\vartheta_{0,\delta}) dx.
\end{split}
\end{equation}
Choosing $h(\vartheta_n)=\frac{1}{1+\vartheta_n}$ in \eqref{4.32}, we obtain
\begin{equation}\label{4.33}
\begin{split}
&\int_0^\tau \int_\Omega \frac{\kappa(\vartheta_n)}{(1+\vartheta_n)^{2}}|\nabla\vartheta_n|^2 dx ds\\
&=\int_\Omega \left((\delta+\varrho_n)H(\vartheta_n)\right)(\tau) dx-\int_\Omega (\delta+\varrho_{0,\delta})H(\vartheta_{0,\delta}) dx\\
&+\delta \int_0^\tau  \int_\Omega \frac{\vartheta_n^3}{1+\vartheta_n} dx ds
-(1-\delta)\int_0^\tau \int_\Omega\frac{\mathbb{S}_n:\nabla\mathbf{u}_n}{1+\vartheta_n} dx ds,
\end{split}
\end{equation}
where by the growth restriction imposed on $\kappa(\vartheta)$ \eqref{1.33}, the term on the left-hand side can be controlled by
\begin{equation*}\label{4.34}
0<C\int_0^\tau \int_\Omega |\nabla\vartheta_n|^2  dxds
\leq\int_0^\tau \int_\Omega \frac{\kappa(\vartheta_n)}{(1+\vartheta_n)^2}|\nabla\vartheta_n|^2 dxds,
\end{equation*}
and hence \eqref{4.39}.
\end{proof}

\

\subsubsection{{\bf Strong convergence of the approximate density $\varrho_n$}}\

In this subsection, we recall a strong convergence result for the density $\varrho_n$.

\begin{lemma}\label{5.1.}{\rm(\cite{DiPerna and Lions,Lions 1996 Incompressible models})}
Assume that the sequence $(\varrho_n, \mathbf{u}_n)$ solves equations (\ref{1.50}a), (\ref{1.50}b) in the sense of distributions, and satisfies the estimates \eqref{4.37} and \eqref{4.38}.

Then we have
\begin{equation}\label{5.10}
\varrho_n\rightarrow \varrho{\rm\,\, in}\,\, C([0,T]; L^p(\Omega))
\end{equation}
for any $1\leq p<\infty$.
Moreover, $\varrho$ satisfies
\begin{equation}\label{5.11}
meas\{x\in\Omega\,|\,\alpha\leq\varrho(t,x)\leq\beta\} \,\,{\rm is \,\,independent\,\, of}\,\,t\geq0,
\end{equation}
for all $0\leq\alpha\leq\beta<\infty$.
\end{lemma}

\

\subsubsection{{\bf Limit in the approximate continuity equation}}\

By \eqref{4.38}, we assume
\begin{equation}\label{5.2}
\mathbf{u}_n\rightharpoonup\mathbf{u}{\rm\,\,weakly\,\,in\,\,}L^2(0,T;H^{1}_0(\Omega)),
\end{equation}
at least for a suitable subsequence.

Therefore, in order to prove that $\{\varrho, \mathbf{u}\}$ solves the continuity equation (\ref{1.50}a) in the sense of distributions, it suffices to show
\begin{equation}\label{5.3}
\varrho_n \mathbf{u}_n\rightarrow\varrho\mathbf{u}\,\,{\rm in}\,\,\mathcal{D}'((0,T)\times\Omega),
\end{equation}
which can be achieved by \eqref{5.10} and \eqref{5.2}.

\

\subsubsection{{\bf Strong convergence of the approximate temperature $\vartheta_n$}}\

Before proving the strong convergence of the approximate temperature $\vartheta_n$, we need the following variant of the Aubin-Lions lemma, which plays an essential role in our proof.

\begin{lemma}\label{5.2.}{\rm (Lemma 6.3 in \cite{Feireisl Dynamics 2004})}
Let $\{\mathbf{v}_n\}_{n=1}^\infty$ be a sequence of functions such that
\begin{equation}\label{5.20}
\mathbf{v}_n {\rm\,\,is \,\,bounded \,\,in}\,\,L^2(0,T; L^q(\Omega))\cap L^\infty(0,T; L^1(\Omega)), \quad{\rm with} \,\,q>6/5,
\end{equation}
furthermore, assume that
\begin{equation}\label{5.21}
\partial_t{\mathbf{v}_n} \geq l_n \,\,{\rm in} \,\,\mathcal{D}'((0,T)\times\Omega),
\end{equation}
where
\begin{equation}\label{5.22}
l_n {\rm\,\,is\,\,bounded \,\, in}\,\,L^1(0,T; W^{-m,r}(\Omega))
\end{equation}
for a certain $m\geq1$, $r>1$.

Then $\{\mathbf{v}_n\}_{n=1}^\infty$ contains a subsequence such that
\begin{equation}\label{5.23}
\mathbf{v}_n\rightarrow \mathbf{v}\,\,{\rm in} \,\,L^2(0,T; H^{-1}(\Omega)).
\end{equation}
\end{lemma}

\

Now we apply Lemma \ref{5.2.} to (\ref{1.50}d) with $\mathbf{v}_n=(\delta+\varrho_n)\vartheta_n$. By estimates \eqref{4.37}, \eqref{4.39} and \eqref{4.42}, we have
\begin{equation}\label{5.24}
(\delta+\varrho_n)\vartheta_n{\rm\,\,is \,\,bounded \,\,in}\,\,L^2((0,T)\times\Omega)\cap L^\infty(0,T; L^1(\Omega)),
\end{equation}
and
\begin{equation}\label{5.25}
\begin{split}
\partial_t((\delta+\varrho_n)\vartheta_n)=l_n \,\,{\rm in}\,\, \mathcal{D}'((0,T)\times\Omega),
\end{split}
\end{equation}
where
\begin{equation}\label{5.26}
l_n=-{\rm div}(\varrho_n\mathbf{u}_n\vartheta_n)
+\triangle\mathcal{K}(\vartheta_n)-\delta\vartheta_n^3+(1-\delta)\mathbb{S}_n:\nabla\mathbf{u}_n.
\end{equation}
Owing to estimates \eqref{4.37}-\eqref{4.40} and the compact imbedding of $L^1(\Omega)$ into $W^{-1,s}(\Omega)$, with $s\in (1,\frac{4}{3})$, we have
\begin{equation}\label{5.27}
l_n \,\,{\rm is\,\,bounded \,\,in}\,\,L^1(0,T; W^{-3,r}(\Omega)), \,\,r>1.
\end{equation}
Therefore, we obtain
\begin{equation}\label{5.28}
(\delta+\varrho_n)\vartheta_n\rightarrow (\delta+\varrho)\vartheta \,\,{\rm in}\,\,L^2(0,T; H^{-1}(\Omega)).
\end{equation}
Thanks to \eqref{4.39}, this yields
\begin{equation}\label{5.31}
(\delta+\varrho_n)\vartheta_n^2\rightarrow (\delta+\varrho)\vartheta^2\,\,{\rm in}\,\,\mathcal{D}'((0,T)\times\Omega),
\end{equation}
which, combined with \eqref{5.10}, implies
\begin{equation}\label{5.32}
\begin{split}
&\int_0^T \int_\Omega (\delta+\varrho)\vartheta_n^2 dxdt\\
&=\int_0^T \int_\Omega [(\delta+\varrho)-(\delta+\varrho_n)]\vartheta_n^2 dxdt+\int_0^T \int_\Omega(\delta+\varrho_n)\vartheta_n^2 dxdt\\
&\rightarrow \int_0^T \int_\Omega (\delta+\varrho)\vartheta^2 dxdt.
\end{split}
\end{equation}
Thus, by \eqref{5.32} we obtain
\begin{equation}\label{5.35}
\vartheta_n\rightarrow\vartheta\,\,{\rm in}\,\,L^2((0,T)\times\Omega).
\end{equation}

\

\subsubsection{{\bf Limit in the approximate momentum equation \eqref{4.0}}}\

Fixing $\eta$ in \eqref{4.0}, multiplying by a test function $\psi\in C_c^\infty(0,T)$ and integrating by parts with respect to t, we obtain
\begin{equation}\label{5.0}
\int_0^T\int_\Omega \partial_t\psi\varrho_n\mathbf{u}_n\cdot\eta dxdt
+\int_0^T\int_\Omega \psi \left(\varrho_n\mathbf{u}_n\otimes\mathbf{u}_n-\mathbb{S}_n-\varepsilon\nabla\mathbf{u}_n\right):\nabla\eta dxdt=0.
\end{equation}
In order to prove (\ref{1.50}c) holds in $\mathcal{D}'((0,T)\times\Omega)$, we pass to the limit for $n\rightarrow\infty$ in \eqref{5.0}.
We focus our attention on the nonlinear terms $\varrho_n\mathbf{u}_n\otimes\mathbf{u}_n$ and $\mathbb{S}_n=\mu(\vartheta_n)(\nabla\mathbf{u}_n+\nabla^{\mathbf{t}}\mathbf{u}_n)$, since the rest terms in \eqref{5.0} can be dealt with easily.

First, similarly as in the previous subsection, applying Lemma \ref{5.2.} to \eqref{5.0} with $\mathbf{v}_n=\varrho_n\mathbf{u}_n$, we have
\begin{equation}\label{5.41}
\varrho_n\mathbf{u}_n\rightarrow\varrho\mathbf{u}\,\,{\rm in}\,\,L^2([0,T];H^{-1}(\Omega)),
\end{equation}
which, combined with \eqref{5.2}, gives
\begin{equation}\label{5.42}
\varrho_n\mathbf{u}_n\otimes\mathbf{u}_n\to \varrho\mathbf{u}\otimes\mathbf{u}\,\,{\rm \,\,in\,\,}\mathcal{D}'((0,T)\times\Omega).
\end{equation}

Then, for the nonlinear term $\mu(\vartheta_n)(\nabla\mathbf{u}_n+\nabla^{\mathbf{t}}\mathbf{u}_n)$, by \eqref{5.35} and the Lipschitz continuity of $\mu(\vartheta)$, we have
\begin{equation}\label{5.420}
\mu(\vartheta_n)\rightarrow\mu(\vartheta)\,\,{\rm in}\,\,L^2((0,T)\times\Omega),
\end{equation}
which with help of \eqref{5.2} yields
\begin{equation}\label{5.43}
\mu(\vartheta_n)(\nabla\mathbf{u}_n+\nabla^{\mathbf{t}}\mathbf{u}_n)\to \mu(\vartheta)(\nabla\mathbf{u}+\nabla^{\mathbf{t}}\mathbf{u})
\,\,{\rm \,\,in\,\,}\mathcal{D}'((0,T)\times\Omega).
\end{equation}

\

\subsubsection{{\bf Limit in the renormalized temperature equation \eqref{4.30}}}\

First, by \eqref{5.35} and the properties of $h(\vartheta)$ in Proposition \ref{2.1.}, we have
\begin{equation}\label{5.47}
H(\vartheta_n)\rightarrow H(\vartheta)\,\,{\rm in}\,\,L^2((0,T)\times\Omega),
\end{equation}
which combined with \eqref{5.10}, yields
\begin{equation}\label{5.48}
(\delta+\varrho_n)H(\vartheta_n)\rightarrow(\delta+\varrho)H(\vartheta)\,\,{\rm in}\,\,L^2(0,T; L^{m}(\Omega)),
\end{equation}
for any $1\leq m<2$. Moreover, by \eqref{5.10}, \eqref{5.2} and \eqref{5.47}, we have
\begin{equation}\label{5.49}
\varrho_nH(\vartheta_n)\mathbf{u}_n\rightharpoonup \varrho H(\vartheta)\mathbf{u}\,\,{\rm weakly\,\,in}\,\,L^1((0,T)\times\Omega).
\end{equation}

Then, thanks to Proposition 2.1 in \cite{Feireisl Dynamics 2004}, we are able to deal with the terms $\mathcal{K}_h(\vartheta_n)$ and $\delta\vartheta_n^3h(\vartheta_n)$.
To be specific, by the property
\begin{equation*}
\lim_{z\rightarrow\infty} h(z)=0,
\end{equation*}
we have
\begin{equation*}
\lim_{z\rightarrow\infty} \frac{\mathcal{K}_h(z)}{\mathcal{K}(z)}=0.
\end{equation*}
By \eqref{4.40} and the growth restriction imposed on $\kappa(\vartheta)$ \eqref{1.33}, we have
\begin{equation*}
\sup_{n\geq1}\int_\Omega \mathcal{K}(\vartheta_n) dy<\infty,
\end{equation*}
 thus,
 Proposition 2.1 in \cite{Feireisl Dynamics 2004} yields
\begin{equation}\label{5.52}
\mathcal{K}_h(\vartheta_n)\rightharpoonup\mathcal{K}_h(\vartheta)\,\,{\rm weakly\,\,in}\,\,L^1((0,T)\times\Omega).
\end{equation}
Similarly, for the term $\delta\vartheta_n^3h(\vartheta_n)$, owing to
\begin{equation*}
\lim_{z\rightarrow\infty} \frac{z^3h(z)}{z^3}=0,
\end{equation*}
and
\begin{equation*}
\sup_{n\geq1}\int_\Omega \vartheta_n^3 dy<\infty,
\end{equation*}
we have
\begin{equation}\label{5.54}
\delta\vartheta_n^3h(\vartheta_n)\rightharpoonup\delta\vartheta^3h(\vartheta)\,\,{\rm weakly\,\,in}\,\,L^1((0,T)\times\Omega).
\end{equation}

Next, owing to $0<\underline{\kappa}\leq\kappa(\vartheta)$ and the non-increasing property of $h$, we have
\begin{equation}\label{5.55}
-\int_0^T\int_\Omega \kappa(\vartheta)h'(\vartheta)|\nabla\vartheta|^2 \varphi dxdt
\leq -\liminf_{n\rightarrow\infty} \int_0^T\int_\Omega \kappa(\vartheta_n)h'(\vartheta_n)|\nabla\vartheta_n|^2 \varphi dxdt,
\end{equation}
for any non-negative function $\varphi\in C_c^\infty((0,T)\times\Omega)$.

Finally, it is crucial to have the following lemma to deal with the term $\mathbb{S}_n:\nabla\mathbf{u}_nh(\vartheta_n)$.
\begin{lemma}\label{5-1}{\rm (\cite{Xianpeng Hu and Dehua Wang})}
Let $g(\vartheta)$ be a bounded, continuous and non-negative function from $[0,\infty)$ to $\mathbb{R}$. Suppose that $\vartheta_n$ and $\mathbf{u}_n$ are two
sequences of functions defined on $\Omega$ satisfying
\begin{equation*}
\vartheta_n\rightarrow\vartheta\,\,a.e.\,\,{\rm in}\,\,\Omega,
\end{equation*}
and
\begin{equation*}
\mathbf{u}_n\rightharpoonup\mathbf{u}\,\,{\rm weakly\,\,in\,\,}H^{1}(\Omega).
\end{equation*}
Then
\begin{equation*}
\int_\Omega g(\vartheta)h(\vartheta)|\nabla\mathbf{u}|^2 dx
\leq \liminf_{n\rightarrow\infty}\int_\Omega g(\vartheta_n)h(\vartheta_n)|\nabla\mathbf{u}_n|^2 dx,
\end{equation*}
where the function $h(\vartheta)$ satisfies \eqref{2.10} and \eqref{2.01}.
In particular,
\begin{equation}\label{5.59}
\int_\Omega \mathbb{S}:\nabla\mathbf{u}h(\vartheta)\varphi  dx
\leq \liminf_{n\rightarrow\infty}\int_\Omega \mathbb{S}_n:\nabla\mathbf{u}_nh(\vartheta_n) \varphi dx,
\end{equation}
for any non-negative function $\varphi\in C_c^\infty((0,T)\times\Omega)$.
\end{lemma}
With this lemma at hand, combining \eqref{5.48}-\eqref{5.59}, we can obtain \eqref{2.8} by letting $n\rightarrow\infty$ in \eqref{4.30}.
We are also able to deduce the energy inequality \eqref{2.11} in Proposition \ref{2.1.} by
letting $n\rightarrow\infty$ in \eqref{4.23}.

\

\section{{\bf Limit passage for $\varepsilon$ tends to zero}}\

In this section we use $(\varrho_\varepsilon, \mathbf{u}_\varepsilon, \vartheta_\varepsilon)$ to denote the weak solutions constructed in Proposition \ref{2.1.}.
The main task is to pass limits to $(\varrho_\varepsilon, \mathbf{u}_\varepsilon, \vartheta_\varepsilon)$  as $\varepsilon$ goes to zero. Note that, for any fixed $\varepsilon>0$, $\sqrt{\varepsilon}\nabla\mathbf{u}_{n}$ is bounded in $L^2((0,T)\times\Omega)$,  which is crucial to show the compactness of weak solutions as $n$ goes to infinity.  However, this estimate is not uniform on $\varepsilon$. This will lead to the loss of compactness of weak solutions. Our alternative way is to show that the temperature-depending viscosity coefficient is bounded below from zero, which can provide the uniform bound of $\nabla\mathbf{u}_{\varepsilon}$ in $L^2((0,T)\times\Omega)$.
With such a bound, we are able to obtain the exact same compactness as in Section 2. Thus, to pass to the limits as $\varepsilon\to 0$, we mainly need to prove the temperature is bounded below from zero.
Therefore, this section will be devoted to show a positive bound from below for the temperature, which is uniform in terms of $\varepsilon>0$ and $\delta>0$.

\

\subsection{{\bf A positive bound from below for the temperature}}\

We first give our result about the positive bound for the temperature $\vartheta_\varepsilon$.
\begin{proposition}\label{6.1.}
Let $(\varrho_\varepsilon, \mathbf{u}_\varepsilon, \vartheta_\varepsilon)$ be a weak solution to the approximate system \eqref{1.50}-\eqref{1.55} in the sense of Proposition \ref{2.1.}. Assume that the initial temperature satisfies the assumptions in \eqref{1.56}, that is,
\begin{equation}\label{5-0}
\vartheta_\varepsilon(0)=\vartheta_{0,\delta}\geq\underline{\vartheta}>0.
\end{equation}
Then there exists a constant $\widetilde{\vartheta}>0$ such that
\begin{equation}\label{5-1}
\vartheta_{\varepsilon}(t,x)\geq\widetilde{\vartheta}>0
\end{equation}
for all $t\in[0,T]$ and almost all $x\in\Omega$.
\end{proposition}

\

\begin{remark}\label{6-1}
We emphasize here that the constant $\widetilde{\vartheta}$ does not depend on the parameters $\varepsilon>0$ and $\delta>0$, which is essential in the limit passage later.
\end{remark}

\

To prove Proposition \ref{6.1.}, we need the following important lemmas.
\begin{lemma}\label{6.2.}{\rm (\cite{Vasseur  Lp estimates})}
Let $U_k$ be a sequence satisfying
\begin{enumerate}[(i).]
\item $0\leq U_0\leq C$;

\item for some constants $A\geq1$, $1<\beta_1<\beta_2$ and $C>0$,
\begin{equation}\label{6.4}
0\leq U_k\leq C\frac{A^k}{K}(U_{k-1}^{\beta_1}+U_{k-1}^{\beta_2}).
\end{equation}
\end{enumerate}
Then there exists some $K_0$ such that for every $K>K_0$, the sequence $U_k$ converges to $0$ when $k$ goes to infinity.
\end{lemma}

\

\begin{lemma}\label{6.20.}{\rm (\cite{Feireisl Dynamics 2004})}
Let $\varrho$ be a non-negative function such that
\begin{equation}\label{4.17}
0<M_1\leq \int_\Omega \varrho dx, \int_\Omega \varrho^\gamma dx  \leq M_2,  \,\,{\rm with}\,\,\gamma>\frac{6}{5}.
\end{equation}
Then there exists a positive constant $C$ depending only on $M_1$, $M_2$ such that
\begin{equation}\label{4.18}
\|\mathbf{v}\|_{H^{1}(\Omega)}\leq C\left(\|\nabla\mathbf{v}\|_{L^2(\Omega)}+\int_\Omega \varrho|\mathbf{v}| dx\right).
\end{equation}
\end{lemma}

\

Our proof is in the spirit of the work of Mellet and Vasseur \cite{Vasseur temperature}, where they first used De Giorgi's method to give a positive bound from below for the temperature.

\

\begin{proof}[Proof of Proposition \ref{6.1.}]
For clarity, we divide the proof into three steps.

\underline{\bf{Step 1.}} Taking $H(\vartheta)=-\int_0^\vartheta h(z)dz$ with $h(z)=\frac{1}{z+\omega}1_{\{z+\omega\leq C\}}$ for some constant $\omega>0$, we have
\begin{equation}\label{6.7}
H(\vartheta_\varepsilon)=
\begin{cases}
-ln(\vartheta_\varepsilon+\omega)+ln\omega,\,\,&{\rm if} \,\,\vartheta_\varepsilon+\omega \leq C,\\
-ln C +ln \omega,\,\,&{\rm if} \,\,\vartheta_\varepsilon+\omega > C.\\
\end{cases}
\end{equation}
Thanks to \eqref{2.8}, the weak solution $(\varrho_\varepsilon, \mathbf{u}_\varepsilon, \vartheta_\varepsilon)$ satisfies the following  temperature inequality in the sense of distributions
\begin{equation}\label{6.700}
\begin{split}
&\partial_t((\delta+\varrho_\varepsilon) H(\vartheta_\varepsilon))+{\rm div}(\varrho_\varepsilon\mathbf{u}_\varepsilon H(\vartheta_\varepsilon))-\triangle\mathcal{K}_h(\vartheta_\varepsilon)
-h'(\vartheta_\varepsilon)\kappa(\vartheta_\varepsilon)|\nabla\vartheta_\varepsilon|^2\\
&\leq\delta\vartheta_\varepsilon^3h(\vartheta_\varepsilon)
-(1-\delta)\mathbb{S}_\varepsilon:\nabla\mathbf{u}_\varepsilon h(\vartheta_\varepsilon),
\end{split}
\end{equation}
with $H(\vartheta)=-\int_0^\vartheta h(z)dz$ and $\mathcal{K}_h(\vartheta)=-\int_0^\vartheta \kappa(z)h(z) dz$.

Then, letting $\phi(\vartheta_\varepsilon)=H(\vartheta_\varepsilon)+ln C-ln \omega=\left[ln\left(\frac{C}{\vartheta_\varepsilon+\omega}\right)\right]_{+}$ and integrating \eqref{6.700} over $(s,t)\times\Omega$ for any $0\leq s\leq t\leq T$, we deduce
\begin{equation}\label{6.8}
\begin{split}
&\int_\Omega \left((\delta+\varrho_\varepsilon)\phi(\vartheta_\varepsilon)\right)(t)dx
-2(1-\delta)\int_s^t\int_\Omega \mu(\vartheta_\varepsilon)
|D(\mathbf{u}_\varepsilon)|^2\phi'(\vartheta_\varepsilon)dxd\tau\\
&\quad+\int_s^t\int_\Omega \phi''(\vartheta_\varepsilon)\kappa(\vartheta_\varepsilon)|\nabla\vartheta_\varepsilon|^2dxd\tau\\
&\leq \int_\Omega \left((\delta+\varrho_\varepsilon)\phi(\vartheta_\varepsilon)\right)(s)dx
-\delta\int_s^t\int_\Omega\vartheta_\varepsilon^3 \phi'(\vartheta_\varepsilon) dxd\tau,
\end{split}
\end{equation}
where $D(\mathbf{u})=\frac{1}{2}\left(\nabla\mathbf{u}+\nabla^{\mathbf{t}}\mathbf{u}\right)$.
Now, introducing a sequence of real numbers
\begin{equation}\label{6.9}
C_k=e^{-M[1-2^{-k}]}\quad \text{for all positive integers }  k,
\end{equation}
where $M$ is a positive number to be chosen later. We define $\phi_{k,\omega}$ as
\begin{equation}\label{6.10}
\phi_{k,\omega}(\vartheta_\varepsilon)=\left[{\rm ln}\left(\frac{C_k}{\vartheta_\varepsilon+\omega}\right)\right]_{+},
\end{equation}
then
\begin{equation}\label{6.11}
\phi_{k,\omega}'(\vartheta_\varepsilon)=-\frac{1}{\vartheta_\varepsilon+\omega} 1_{\{\vartheta_\varepsilon+\omega\leq C_k\}},
\end{equation}
\begin{equation}\label{6.12}
\phi_{k,\omega}''(\vartheta_\varepsilon)\geq \frac{1}{(\vartheta_\varepsilon+\omega)^2} 1_{\{\vartheta_\varepsilon+\omega\leq C_k\}}.
\end{equation}
Next define $U_{k,\,\omega}$ by
\begin{small}
\begin{equation}\label{6.13}
\begin{split}
U_{k,\,\omega}&:=\sup_{T_k\leq t\leq T}\left(\int_\Omega (\delta+\varrho_\varepsilon) \phi_{k,\omega}(\vartheta_\varepsilon)dx\right)
 +2(1-\delta)\int_{T_k}^T\int_\Omega\frac{\mu(\vartheta_\varepsilon)}{\vartheta_\varepsilon+\omega}
1_{\{\vartheta_\varepsilon+\omega\leq C_k\}}|D(\mathbf{u}_\varepsilon)|^2 dxdt\\
&\quad\quad\quad +\int_{T_k}^T\int_\Omega\frac{\kappa(\vartheta_\varepsilon)}{(\vartheta_\varepsilon+\omega)^2}
1_{\{\vartheta_\varepsilon+\omega\leq C_k\}}|\nabla\vartheta_\varepsilon|^2 dxdt,
\end{split}
\end{equation}
\end{small}where $\{T_k\}$ is a sequence of non-negative numbers.
Note that $U_{k,\omega}$ depends on $\varepsilon$, $\delta$ and $\omega$, that is, $U_{k,\,\omega}=U_{k, \varepsilon, \delta, \omega}$, and for convenience, we
still write it as $U_{k,\omega}$.

Assuming $T_k=0$ for all $k\in \mathbb{N}$,  from \eqref{6.8} and \eqref{6.13}, we claim that
\begin{equation}\label{6.14}
\begin{split}
U_{k,\omega}\leq & \int_\Omega \left(\delta+\varrho_{0,\delta}\right)\phi_{k,\omega}(\vartheta_{0,\delta})dx
+\delta\int_{T_{k-1}}^T\int_\Omega \frac{\vartheta_\varepsilon^3}{\vartheta_\varepsilon+\omega}1_{\{\vartheta_\varepsilon+\omega\leq C_k\}}dxdt.
\end{split}
\end{equation}
In fact, taking $0\leq T_{k-1}\leq s\leq T_{k}\leq t\leq T$ in \eqref{6.8} and using \eqref{6.11} and \eqref{6.12}, one gets
\begin{equation}\label{6.16}
\begin{split}
&\int_\Omega \left((\delta+\varrho_\varepsilon) \phi_{k,\omega}(\vartheta_\varepsilon)\right)(t)dx\\
&+2(1-\delta)\int_{T_k}^t\int_\Omega\frac{\mu(\vartheta_\varepsilon)}{\vartheta_\varepsilon+\omega}
1_{\{\vartheta_\varepsilon+\omega\leq C_k\}}|D(\mathbf{u}_\varepsilon)|^2 dxd\tau\\
& +\int_{T_k}^t\int_\Omega\frac{\kappa(\vartheta_\varepsilon)}{(\vartheta_\varepsilon+\omega)^2}
1_{\{\vartheta_\varepsilon+\omega\leq C_k\}}|\nabla\vartheta_\varepsilon|^2 dxd\tau\\
&\leq \int_\Omega \left((\delta+\varrho_\varepsilon) \phi_{k,\omega}(\vartheta_\varepsilon)\right)(s)dx
+\delta\int_{T_{k-1}}^T\int_\Omega \frac{\vartheta_\varepsilon^3}{\vartheta_\varepsilon+\omega}1_{\{\vartheta_\varepsilon+\omega\leq C_k\}}dxd\tau.
\end{split}
\end{equation}
Taking the supremum over $t\in[T_k, T]$ on both sides of \eqref{6.16}, one deduces that
\begin{equation}\label{6.17}
\begin{split}
U_{k,\omega} &\leq \int_\Omega \left((\delta+\varrho_\varepsilon) \phi_{k,\omega}(\vartheta_\varepsilon)\right)(s)dx
+\delta\int_{T_{k-1}}^T\int_\Omega \frac{\vartheta_\varepsilon^3}{\vartheta_\varepsilon+\omega}1_{\{\vartheta_\varepsilon+\omega\leq C_k\}}dxdt.
\end{split}
\end{equation}
If $T_k=0$ for all $k\in N$, then $s=0$ in \eqref{6.17}, thus we get \eqref{6.14}.

\

\underline{\bf{Step 2.}}
In this step, we prove that the second term on the right-hand side of \eqref{6.14} can be controlled by $U_{k-1,\omega}^\gamma$ for some $\gamma>1$.
More precisely, we claim that
\begin{equation}\label{6.18}
\begin{split}
&\delta\int_{T_{k-1}}^T\int_\Omega \frac{\vartheta_\varepsilon^3}{\vartheta_\varepsilon+\omega}1_{\{\vartheta_\varepsilon+\omega\leq C_k\}}dxdt
\leq C\frac{2^{k\alpha}}{M^{\alpha}}U_{k-1,\omega}^{\gamma},
\end{split}
\end{equation}
for some $\gamma>1$, where the constant $C$ is independent of $\varepsilon,\delta>0$.

Indeed, if $\vartheta_\varepsilon+\omega\leq C_k$, we have
\begin{equation}\label{6.19}
\frac{\vartheta_\varepsilon^3}{\vartheta_\varepsilon+\omega}\leq 1,\,\,{\rm for\,\,any}\,\,\omega>0,
\end{equation}
 by taking $M$ large enough such that $C_k$ is small enough, and
\begin{equation*}
\phi_{k-1,\omega}(\vartheta_\varepsilon)=\left[{\rm ln}\left(\frac{C_{k-1}}{\vartheta_\varepsilon+\omega}\right)\right]_{+}
\geq {\rm ln}\frac{C_{k-1}}{C_k},
\end{equation*}
which implies
\begin{equation}\label{6.21}
1_{\{\vartheta_\varepsilon+\omega\leq C_k\}}\leq \left[{\rm ln}\frac{C_{k-1}}{C_k}\right]^{-\alpha}\phi_{k-1,\omega}(\vartheta_\varepsilon)^{\alpha}, \,\,{\rm for \,\,any}\,\,\alpha>0.
\end{equation}
Taking \eqref{6.19} and \eqref{6.21} into account, we deduce
\begin{equation}\label{6.22}
\begin{split}
&\delta\int_{T_{k-1}}^T\int_\Omega \frac{\vartheta_\varepsilon^3}{\vartheta_\varepsilon+\omega}1_{\{\vartheta_\varepsilon+\omega\leq C_k\}}dxdt\\
&\leq \delta^{1-\beta}\left[{\rm ln}\frac{C_{k-1}}{C_k}\right]^{-\alpha}\int_{T_{k-1}}^T\int_\Omega (\delta+\varrho_\varepsilon)^\beta
\phi_{k-1,\omega}(\vartheta_\varepsilon)^{\alpha}dxdt\\
&\leq C \delta^{1-\beta}\left[{\rm ln}\frac{C_{k-1}}{C_k}\right]^{-\alpha}T^{1/p'}|\Omega|^{1/q'}
\|(\delta+\varrho_\varepsilon)^\beta\phi_{k-1,\omega}(\vartheta_\varepsilon)^{\alpha}\|_{L^p(T_{k-1},T; L^q(\Omega))},\\
\end{split}
\end{equation}
where $\frac{1}{p}+\frac{1}{p'}=1$ and $\frac{1}{q}+\frac{1}{q'}=1$.
For the last term in \eqref{6.22}, we have
\begin{equation}\label{6.23}
\begin{split}
&\|(\delta+\varrho_\varepsilon)^\beta\phi_{k-1,\omega}(\vartheta_\varepsilon)^{\alpha}\|_{L^p(T_{k-1},T; L^q(\Omega))}\\
&= \|\left((\delta+\varrho_\varepsilon)\phi_{k-1,\omega}(\vartheta_\varepsilon)\right)^{\beta/\alpha}
\phi_{k-1,\omega}(\vartheta_\varepsilon)^{1-\beta/\alpha}\|_{L^{p \alpha}(T_{k-1},T; L^{q \alpha}(\Omega))}^{\alpha}\\
&\leq \|\left((\delta+\varrho_\varepsilon)\phi_{k-1,\omega}(\vartheta_\varepsilon)\right)^{\beta/\alpha}\|
_{L^\infty(T_{k-1},T;L^{\alpha/\beta}(\Omega))}^{\alpha}\\
&\quad\quad\quad\quad \|\phi_{k-1,\omega}(\vartheta_\varepsilon)^{1-\beta/\alpha}\|_{L^{\frac{2}{1-\beta/\alpha}}(T_{k-1},T; L^{\frac{6}{1-\beta/\alpha}}(\Omega))}^{\alpha}\\
&= \|(\delta+\varrho_\varepsilon)\phi_{k-1,\omega}(\vartheta_\varepsilon)\|^{\beta}_{L^\infty(T_{k-1},T;L^1(\Omega))}
\|\phi_{k-1,\omega}(\vartheta_\varepsilon)\|_{L^2(T_{k-1},T; L^6(\Omega))}^{\alpha-\beta}\\
&\leq \|(\delta+\varrho_\varepsilon)\phi_{k-1,\omega}(\vartheta_\varepsilon)\|^{\beta}_{L^\infty(T_{k-1},T;L^1(\Omega))}\\
&\quad\quad\quad\quad     \left(\|\varrho_\varepsilon \phi_{k-1,\omega}(\vartheta_\varepsilon)\|_{L^\infty(T_{k-1},T;L^1(\Omega))}
+\|\nabla\phi_{k-1,\omega}(\vartheta_\varepsilon)\|_{L^2((T_{k-1},T)\times\Omega)}\right)^{\alpha-\beta}\\
&\leq C U_{k-1,\omega}^{\beta}\left(U_{k-1,\omega}+U_{k-1,\omega}^{1/2}\right)^{\alpha-\beta}\\
&\leq C \left(U_{k-1,\omega}^{\alpha}+U_{k-1,\omega}^{\frac{\alpha+\beta}{2}}\right),
\end{split}
\end{equation}
where we used Lemma \ref{6.20.} in the third inequality from below, the growth restriction imposed on $\kappa(\vartheta)$ \eqref{1.33} in the second inequality from below, and the coefficients $p$, $q$, $\alpha$ and $\beta$ satisfy
\begin{equation*}\label{6.24}
\begin{split}
&\frac{1}{p \alpha}=\frac{1-\beta/\alpha}{2},\;\frac{1}{q \alpha}=\frac{\beta}{\alpha}+\frac{1-\beta/\alpha}{6}.
\end{split}
\end{equation*}
Substituting \eqref{6.23} into \eqref{6.22}, we have
\begin{equation}\label{6.35}
\begin{split}
\delta\int_{T_{k-1}}^T\int_\Omega \frac{\vartheta_\varepsilon^3}{\vartheta_\varepsilon+\omega}1_{\{\vartheta_\varepsilon+\omega\leq C_k\}}dxdt
\leq C \delta^{1-\beta}\left[{\rm ln}\frac{C_{k-1}}{C_k}\right]^{-\alpha}\left(U_{k-1,\omega}^{\alpha}+U_{k-1,\omega}^{\frac{\alpha+\beta}{2}}\right).
\end{split}
\end{equation}
Then, by \eqref{6.9}, we have
\begin{equation}\label{6.39}
\left[{\rm ln}\frac{C_{k-1}}{C_k}\right]^{-\alpha}=\frac{2^{k\alpha}}{M^{\alpha}}.
\end{equation}
Meanwhile, we can choose $\beta<1$, $\alpha>1$ such that
\begin{equation}\label{6.400}
\gamma:=\min\left(\frac{\alpha+\beta}{2},\alpha\right)>1,
\end{equation}
and
\begin{equation}\label{6.401}
\delta^{1-\beta}\leq1.
\end{equation}
Combining \eqref{6.35}-\eqref{6.401} together, we obtain \eqref{6.18}.

\

\underline{\bf{Step 3.}}
We are now ready to complete the proof of Proposition \ref{6.1.}.
By assumption $\vartheta_{0,\delta}\geq\underline{\vartheta}>0$, choosing $M$ large enough such that $e^{-M/2}<\underline{\vartheta}$, we have for any $\omega>0$
\begin{equation}\label{6.37}
\phi_{k,\omega}(\vartheta_{0,\delta})=\left[{\rm ln}\left(\frac{e^{-M[1-2^{-k}]}}{\vartheta_{0,\delta}+\omega}\right)\right]_{+}=0.
\end{equation}
Substituting \eqref{6.18} and \eqref{6.37} into \eqref{6.14}, we obtain
\begin{equation}\label{6.402}
U_{k,\omega}\leq C\frac{2^{k\alpha}}{M^{\alpha}}U_{k-1,\omega}^{\gamma}\,\,{\rm with}\,\,\gamma>1.
\end{equation}
Thanks to Lemma \ref{6.2.}, for M large enough (independently on $\varepsilon$, $\delta$ and $\omega$), we have
\begin{equation}\label{6.43}
\lim_{k\rightarrow \infty} U_{k,\omega}=0,
\end{equation}
which combined with the definition of $U_{k,\omega}$ \eqref{6.13} yields
\begin{equation}\label{6.44}
\int_0^T\int_\Omega \kappa(\vartheta_\varepsilon)\left|\nabla\left[{\rm ln}\frac{e^{-M}}{\vartheta_\varepsilon+\omega}\right]_{+}\right|^2 dxdt=0,
\end{equation}
and
\begin{equation}\label{6.45}
\int_\Omega (\delta+\varrho_\varepsilon)\left[{\rm ln}\frac{e^{-M}}{\vartheta_\varepsilon+\omega}\right]_{+} dx=0.
\end{equation}
By \eqref{1.33} and \eqref{6.44}, we obtain
\begin{equation}\label{6.46}
\left[{\rm ln}\frac{e^{-M}}{\vartheta_\varepsilon+\omega}\right]_{+} \,\,{\rm is\,\,constant \,\,in\,\,\Omega\,\,for \,\,all\,\,}t\in[0,T],
\end{equation}
and with help of \eqref{6.45}, this implies
\begin{equation*}\label{6.47}
\left[{\rm ln}\frac{e^{-M}}{\vartheta_\varepsilon+\omega}\right]_{+}=0.
\end{equation*}
This yields
\begin{equation*}\label{6.48}
\vartheta_\varepsilon+\omega\geq e^{-M}
\end{equation*}
for any $\omega>0$, which completes our proof.
\end{proof}

\

\subsection{{\bf Estimates independent of $\varepsilon>0$}}\

By virtue of Proposition \ref{6.1.} and the assumptions imposed on $\mu(\vartheta)$ in Theorem \ref{1.2.}, we can obtain a positive constant $\underline{\mu}$ independent of $\varepsilon,\delta>0$ such that
\begin{equation}\label{6.50}
\mu(\vartheta_\varepsilon)\geq \underline{\mu}>0.
\end{equation}
This, combined with the energy inequality \eqref{2.11}, yields
\begin{equation}\label{6.52}
\nabla\mathbf{u}_\varepsilon \,\,{\rm is \,\,bounded \,\,in \,\,}L^2((0,T)\times\Omega)
\end{equation}
by a positive constant independent of $\varepsilon>0$. Thanks to Poincar\'{e}'s inequality, we have
\begin{equation}\label{6.53}
\|\mathbf{u}_\varepsilon\|_{L^2(0,T; H_{0}^{1}(\Omega))}\leq C,
\end{equation}
where $C$ is independent of $\varepsilon>0$. Thus, we are able to get the following uniform bounds for $(\rho_{\varepsilon},\mathbf{u}_{\varepsilon}, \vartheta_{\varepsilon})$ as in Section 2.
\begin{proposition}\label{6.3.}
For fixed $\delta>0$, under the hypotheses of Theorem \ref{1.2.} and Proposition \ref{2.1.}, we have
\begin{equation}\label{6.54}
\|\varrho_\varepsilon\|_{L^\infty((0,T)\times\Omega)}\leq C,
\end{equation}
\begin{equation}\label{6.55}
\|\mathbf{u}_\varepsilon\|_{L^2(0,T; H^{1}_0(\Omega))}\leq C,
\end{equation}
\begin{equation}\label{6.56}
\|\vartheta_\varepsilon\|_{L^2(0,T; H^{1}(\Omega))}\leq C,
\end{equation}
\begin{equation}\label{6.57}
\|\vartheta_\varepsilon\|_{L^3((0,T)\times\Omega)}\leq C,
\end{equation}
\begin{equation}\label{6.58}
\|\sqrt{\varrho}_\varepsilon\mathbf{u}_\varepsilon\|_{L^\infty(0,T; L^2(\Omega))}\leq C,
\end{equation}
\begin{equation}\label{6.59}
\|\varrho_\varepsilon\vartheta_\varepsilon\|_{L^\infty(0,T; L^1(\Omega))}\leq C,
\end{equation}
where all constants $C$ are independent of $\varepsilon>0$.
\end{proposition}
Thanks to \eqref{6.54}-\eqref{6.59}, we are able to derive the same compactness structure for $(\varrho_{\varepsilon},\mathbf{u}_{\varepsilon},\vartheta_\varepsilon)$ as $(\varrho_n,\mathbf{u}_n, \vartheta_n)$. Thus, we can show the following proposition by passing to the limits as $\varepsilon\to0.$

\begin{proposition}\label{6.4.}
For fixed $\delta>0$, under the hypotheses of Theorem \ref{1.2.} and Proposition \ref{2.1.},
the initial-boundary value problem \eqref{1.1}-\eqref{1.6} with the parameter $\delta>0$ admits an approximate solution $(\varrho,\mathbf{u},\vartheta,P)$, which is also the limit of the weak solution to \eqref{1.50}-\eqref{1.55} when $\varepsilon\rightarrow0$,
satisfying
\begin{enumerate}[(i).]

\item the density $\varrho\geq0$ and the velocity $\mathbf{u}$ satisfy
\begin{equation*}
\begin{split}
&\varrho\in L^\infty((0,T)\times\Omega)\cap C([0,T];L^p(\Omega)),\, 1\leq p<\infty,\\
&\mathbf{u}\in L^2(0,T; H^1_0(\Omega)),\quad\sqrt{\varrho}\mathbf{u}\in L^\infty(0,T; L^2(\Omega)),
\end{split}
\end{equation*}
and the temperature $\vartheta\geq 0$ satisfies
\begin{equation*}
\vartheta \in L^2(0,T; H^{1}(\Omega))\cap L^3((0,T)\times\Omega),\quad \varrho\vartheta\in L^\infty(0,T; L^1(\Omega));
\end{equation*}

\item the equations (\ref{1.1}a)-(\ref{1.1}c) hold in $\mathcal{D}'((0,T)\times\Omega)$;

\item  $\vartheta(t,\cdot)\rightarrow \vartheta_{0,\delta}$,\quad $\varrho\mathbf{u}(t,\cdot)\rightarrow \mathbf{m}_{0,\delta}$ in $\mathcal{D}'(\Omega)$ as $t\rightarrow 0^+$;

\item the renormalized temperature inequality holds in the sense of distributions, that is,
\begin{equation}\label{6.67}
\begin{split}
&\int_0^T\int_\Omega (\delta+\varrho)H(\vartheta)\partial_t\varphi dxdt\\
&\quad+\int_0^T\int_\Omega \left(\varrho H(\vartheta)\mathbf{u}\cdot\nabla\varphi
+\mathcal{K}_h(\vartheta)\triangle\varphi-\delta\vartheta^3 h(\vartheta)\varphi\right) dxdt\\
&\leq\int_0^T\int_\Omega \left((\delta-1)\mathbb{S}:\nabla\mathbf{u}h(\vartheta)+h'(\vartheta)\kappa(\vartheta)|\nabla\vartheta|^2 \right)\varphi dxdt\\
&\quad-\int_\Omega(\delta+\varrho_{0,\delta})H(\vartheta_{0,\delta})\varphi(0)dx,
\end{split}
\end{equation}
for any $\varphi\in C_c^\infty([0,T]\times\Omega)$ satisfying
\begin{equation}\label{6.68}
\varphi\geq 0,\,\,\varphi(T,\cdot)=0,\,\,\nabla\varphi\cdot \mathbf{n}|_{\partial\Omega}=0,
\end{equation}
where $H(\vartheta)=\int_0^\vartheta h(z) dz$ and $\mathcal{K}_h(\vartheta)=\int_0^\vartheta \kappa(z)h(z) dz$,
with non-increasing $h\in C^2([0,\infty))$ satisfying
\begin{equation*}
0<h(0)<\infty,\,\,h\,\, \lim_{z\rightarrow\infty}h(z)=0.\\
\end{equation*}
and
\begin{equation*}
h''(z)h(z)\geq 2(h'(z))^2\,\,for \,\,all\,\,z\geq0;
\end{equation*}

\item the energy inequality holds, that is, for a.e. $t\in(0,T)$,
\begin{equation}\label{6.70}
\begin{split}
&\int_\Omega\left(\frac{1}{2}\varrho|\mathbf{u}|^2+(\delta+\varrho)\vartheta\right)(t) dx
+\delta\int_0^t\int_\Omega \mathbb{S}:\nabla\mathbf{u}+\vartheta^3dxds\\
&\leq \int_\Omega\frac{1}{2}\frac{|\mathbf{m}_{0,\delta}|^2}{\varrho_{0,\delta}}
+(\delta+\varrho_{0,\delta})\vartheta_{0,\delta} dx.
\end{split}
\end{equation}
\end{enumerate}
\end{proposition}

\

\section{{\bf Limit passage for $\delta$ tends to zero}}\

The final step is to recover a weak solution to the initial-boundary value problem \eqref{1.1}-\eqref{1.6} by passing to the limit as $\delta\to 0$.
In this section, we use $(\varrho_\delta, \mathbf{u}_\delta, \vartheta_\delta)$ to denote the weak solutions constructed in Proposition \ref{6.4.}.
Note that the positive below bound of the temperature in Proposition \ref{6.1.} does not depend on $\delta>0$ , so we have
\begin{equation}\label{7.00}
\mu(\vartheta_\delta)\geq \underline{\mu}>0
\end{equation}
for some positive constant $\underline{\mu}$ independent of $\delta>0$ in this whole section.

\

\subsection{\bf Estimates independent of $\delta>0$}\

Observe that estimates for $\varrho_\delta$ are similar as in the previous sections, and estimates for $\mathbf{u}_\delta$ can be deduced after some calculations, thus our main task in this section is to deal with
terms related to $\vartheta_\delta$.
For convenience, in the rest of this section, we denote $C$ a generic positive constant independent of $\delta>0$.

First, by \eqref{initial} and the energy inequality \eqref{6.70}, we have the following estimates
\begin{equation}\label{7.1}
\|\sqrt{\varrho_\delta}\mathbf{u}_\delta\|_{L^\infty(0,T; L^2(\Omega))}\leq C,
\end{equation}
\begin{equation}\label{7.2}
\|(\delta+\varrho_\delta)\vartheta_\delta\|_{L^\infty(0,T; L^1(\Omega))}\leq C,
\end{equation}
\begin{equation}\label{7.3}
\delta \int_0^T\int_\Omega \mathbb{S}_\delta:\nabla\mathbf{u}_\delta dxdt\leq C,
\end{equation}
\begin{equation}\label{7.4}
\delta \int_0^T\int_\Omega \vartheta_\delta^3 dxdt\leq C.
\end{equation}

Then, taking $\varphi(t,x)=\psi(t)$ satisfying $0\leq\psi\leq1$, $\psi\in C_c^\infty(0,T)$ and $h(\vartheta)=\frac{1}{(1+\vartheta)^l}$ with $0<l<1$ in \eqref{6.67}, we have
\begin{equation}\label{7.5}
\begin{split}
&\int_0^T\int_\Omega \left(\frac{1-\delta}{(1+\vartheta_\delta)^l}\mathbb{S}_\delta:\nabla\mathbf{u}_\delta
+l\frac{\kappa(\vartheta_\delta)}{(1+\vartheta_\delta)^{l+1}}|\nabla\vartheta_\delta|^2\right)\psi dxdt\\
&\leq \delta\int_0^T\int_\Omega \frac{\vartheta_\delta^{3}}{(1+\vartheta_\delta)^l}\psi dxdt
-\int_0^T\int_\Omega (\delta+\varrho_{\delta})H(\vartheta_{\delta})\partial_t\psi dxdt,
\end{split}
\end{equation}
with $H(\vartheta)=\int_0^\vartheta \frac{1}{(1+z)^l} dz$, which combined with estimates \eqref{7.2}-\eqref{7.4} implies
\begin{equation}\label{7.6}
\begin{split}
\int_0^T\int_\Omega \left(\frac{\mathbb{S}_\delta:\nabla\mathbf{u}_\delta}{(1+\vartheta_\delta)^l}
+l\frac{\kappa(\vartheta_\delta)}{(1+\vartheta_\delta)^{l+1}}|\nabla\vartheta_\delta|^2\right)\psi dxdt
\leq C
\end{split}
\end{equation}
for some constant $C$ independent of $\delta>0$.
Letting $l\rightarrow 0$ in \eqref{7.6}, we obtain
\begin{equation}\label{7.7}
\int_0^T\int_\Omega\mathbb{S}_\delta:\nabla\mathbf{u}_\delta dxdt\leq C.
\end{equation}
This with help of \eqref{7.00} and Poincar\'{e}'s inequality, yields
\begin{equation}\label{7.007}
\|\mathbf{u}_\delta\|_{L^2(0,T; H^{1}_0(\Omega))}\leq C.
\end{equation}
In addition, for fixed $0<l<1$, by virtue of \eqref{7.6} and the growth restriction imposed on $\kappa(\vartheta)$ \eqref{1.33}, we obtain
\begin{equation}\label{7.1000}
\|\nabla\vartheta_\delta^{\frac{3-l}{2}}\|_{L^2((0,T)\times\Omega)}\leq C(l),
\end{equation}
together with
\begin{equation}\label{7.100}
\|\nabla\vartheta_\delta\|_{L^2((0,T)\times\Omega)}\leq C(l),
\end{equation}
with the constant $C(l)$ depending on $l\in(0,1)$.
Thanks to Lemma \ref{6.20.}, estimates \eqref{7.2} and \eqref{7.100} yield
\begin{equation}\label{7.101}
\|\vartheta_\delta\|_{L^2(0,T;H^1(\Omega))}\leq C(l).
\end{equation}
Bootstraping \eqref{7.1000} and \eqref{7.101}, we deduce
\begin{equation}\label{7.10}
\|\vartheta_\delta^{\frac{3-l}{2}}\|_{L^2(0,T;H^1(\Omega))}\leq C(l).
\end{equation}
Combining \eqref{7.2} with \eqref{7.10} and thanks to the interpolation inequality, we deduce for a certain $p>1$ and a small positive number $\omega$
\begin{equation}\label{7.11}
\vartheta_\delta^3\,\,{\rm is \,\,bounded \,\,in}\,\,L^p\left(\{\varrho_\delta(t,x)\geq\omega>0\}\right)
\end{equation}
by a positive constant independent of $\delta>0$.

Putting all estimates independent of $\delta>0$ together, we have the following result.
\begin{proposition}\label{7.1.}
Under the hypotheses of Theorem \ref{1.2.} and Proposition \ref{2.1.}, we have
\begin{equation}\label{7.12}
\|\varrho_\delta\|_{L^\infty((0,T)\times\Omega)}\leq C,
\end{equation}
\begin{equation}\label{7.13}
\|\mathbf{u}_\delta\|_{L^2(0,T; H^{1}_0(\Omega))}\leq C,
\end{equation}
\begin{equation}\label{7.14}
\|\vartheta_\delta^{\frac{3-l}{2}}\|_{L^2(0,T; H^{1}(\Omega))}\leq C(l),\,\,{\rm with \,\,}l\in(0,1],
\end{equation}
\begin{equation}\label{7.15}
\|\vartheta_\delta^3\|_{L^p\left(\{\varrho_\delta(t,x)\geq\omega>0\}\right)}\leq C, \,\,{\rm for \,\,some\,\,}p>1.
\end{equation}
\end{proposition}

By virtue of \eqref{7.12} and \eqref{7.13}, we can obtain the same compactness result for $(\varrho_\delta,\mathbf{u}_\delta)$ as $(\varrho_n,\mathbf{u}_n)$. Thus, in the rest of this section, we focus our attention on the compactness of $\vartheta_\delta$.

\

\subsection{\bf Strong convergence of the temperature $\vartheta_\delta$}\

Similarly as the process in obtaining the strong convergence of $\vartheta_n$, the main difficulty is to prove
\begin{equation}\label{7.20}
(\delta+\varrho_\delta)\vartheta_\delta\rightarrow \varrho \overline{\vartheta}
\,\,{\rm in}\,\,L^2(0,T; H^{-1}(\Omega)), \,\,{\rm as}\,\,\delta\rightarrow0.
\end{equation}
In order to obtain \eqref{7.20}, by virtue of Lemma \ref{5.2.} and estimates \eqref{7.12}-\eqref{7.15}, it suffices to prove
\begin{equation}\label{7.22}
\vartheta_\delta\,\,{\rm is \,\,bounded \,\,in \,\,}L^{3}\left(\{\varrho_\delta(x,t)<\omega\}\right)
\end{equation}
by a positive constant independent of $\delta>0$, with $\omega$ being a sufficiently small positive number.

As in \cite{Feireisl Dynamics 2004, Feireisl On the motion 2004}, \eqref{7.22} can be achieved by taking the test function
\begin{equation*}
\varphi(t,x)=\psi(t)(\eta(t,x)-\underline{\eta}),\,\,0\leq\psi\leq1,\,\,\psi\in C_c^\infty(0,T),
\end{equation*}
in \eqref{6.67},
where
\[\underline{\eta}=\inf_{t\in[0,T],x\in\Omega}\eta,\]
and for each $t\in[0,T]$, $\eta=\eta_\delta$ is the unique solution of the following Neumann problem
\begin{equation*}
\begin{cases}
\triangle\eta_\delta(t)=B(\varrho_\delta(t))-\frac{1}{|\Omega|}\int_\Omega B(\varrho_\delta(t))dx\,\,{\rm in}\,\,\Omega,\\
\nabla\varrho_\delta\cdot\mathbf{n}=0\,\,{\rm on}\,\,\partial\Omega,\\
\int_\Omega\eta_\delta(t)dx=0,
\end{cases}
\end{equation*}
with $B\in C^\infty(\mathbb{R})$ non-increasing satisfying
\begin{equation*}
B(z)=
\begin{cases}
0, & {\rm if}\,\, z\leq \omega,\\
-1, & {\rm if}\,\, z\geq 2\omega.
\end{cases}
\end{equation*}
Since this process is similar as that in \cite{Feireisl Dynamics 2004, Feireisl On the motion 2004}, thus we omit the details here.

Combining \eqref{7.14} and \eqref{7.20}, we have
\begin{equation}\label{7.26}
\vartheta_\delta\rightarrow\overline{\vartheta}\,\,{\rm in}\,\,L^2(\{\varrho>0\}).
\end{equation}

Taking \eqref{7.12}-\eqref{7.26} into account, we can perform the limit $\delta\rightarrow0$ as in Section 2.
For convenience, we only give the details of the limit passage in the renormalized temperature inequality.

\

\subsection{\bf Limit in the renormalized temperature inequality \eqref{6.67}}\

First, for fixed $h$, passing to the limit $\delta\rightarrow0$ in the same way as in Section 2 for the renormalized temperature inequality \eqref{6.67}, we obtain
\begin{equation}\label{7.30}
\begin{split}
&\int_0^T\int_\Omega \left(\varrho H(\bar{\vartheta})\partial_t\varphi
+\varrho\mathbf{u} H(\bar{\vartheta})\cdot\nabla\varphi
+\overline{\mathcal{K}_h(\vartheta)}\triangle\varphi \right) dxdt\\
&\leq -\int_0^T\int_\Omega h(\bar{\vartheta})\mathbb{S}:\nabla\mathbf{u} \varphi dxdt
-\int_\Omega \varrho_{0}H(\vartheta_{0})\varphi(0)dx,
\end{split}
\end{equation}
where
\begin{equation}\label{7.31}
\varrho\overline{\mathcal{K}_h(\vartheta)}=\varrho\mathcal{K}_h(\bar{\vartheta}),
\end{equation}
and
\begin{equation}\label{7.32}
\varrho\mathbb{S}=\varrho\mu(\bar{\vartheta})(\nabla\mathbf{u}+\nabla^{\mathbf{t}}\mathbf{u}).
\end{equation}

Next, taking
\begin{equation}\label{7.33}
h(\vartheta)=\frac{1}{(1+\vartheta)^l},\,\,0<l<1
\end{equation}
in \eqref{7.30}, letting $l\rightarrow 0$ and using the monotone convergence theorem, we have
\begin{equation}\label{7.34}
\begin{split}
&\int_0^T\int_\Omega \left(\varrho \bar{\vartheta}\partial_t\varphi
+\varrho\mathbf{u} \bar{\vartheta}\cdot\nabla\varphi
+\overline{\mathcal{K}(\vartheta)}\triangle\varphi \right) dxdt\\
&\leq -\int_0^T\int_\Omega \mathbb{S}:\nabla\mathbf{u} \varphi dxdt-\int_\Omega \varrho_{0} \vartheta_{0}\varphi(0)dx,
\end{split}
\end{equation}
where
\begin{equation}\label{7.35}
\varrho\overline{\mathcal{K}(\vartheta)}=\varrho\mathcal{K}(\bar{\vartheta}),
\end{equation}
and
\begin{equation}\label{7.36}
\varrho\mathbb{S}=\varrho\mu(\bar{\vartheta})(\nabla\mathbf{u}+\nabla^{\mathbf{t}}\mathbf{u}).
\end{equation}

Finally, denote
\begin{equation}\label{7.37}
\vartheta=\mathcal{K}^{-1}\left(\overline{\mathcal{K}(\vartheta)}\right).
\end{equation}
We observe the new function $\vartheta$ satisfies
\begin{equation}\label{7.39}
\varrho\bar{\vartheta}=\varrho\vartheta\,\,a.e.\,\,{\rm in}\,\,(0,T)\times\Omega.
\end{equation}
Therefore, we have
\begin{equation}\label{7.40}
\begin{split}
&\int_0^T\int_\Omega \left(\varrho \vartheta\partial_t\varphi
+\varrho\mathbf{u} \vartheta\cdot\nabla\varphi
+\mathcal{K}(\vartheta)\triangle\varphi \right) dxdt\\
&\leq -\int_0^T\int_\Omega \mathbb{S}:\nabla\mathbf{u} \varphi dxdt-\int_\Omega \varrho_{0} \vartheta_{0}\varphi(0)dx,
\end{split}
\end{equation}
with
\begin{equation}\label{7.36}
\varrho\mathbb{S}=\varrho\mu(\vartheta)(\nabla\mathbf{u}+\nabla^{\mathbf{t}}\mathbf{u}),
\end{equation}
which is exact the temperature inequality \eqref{1.16} in Definition \ref{1.1.}.

\


\end{document}